\documentclass[preprint,10pt,3p]{elsarticle}
\usepackage{amsmath,amsthm,amsfonts,amssymb}
\usepackage{xcolor}

\usepackage{hyperref}

\newcommand{\rg}{{\mathscr{R}}}
\usepackage{algorithm}
\newcommand{\nl}{{\mathscr{N}}}

\usepackage{algorithmic}
\usepackage{amsfonts}
\usepackage{graphicx}
\usepackage{mathrsfs}
\usepackage{epstopdf}
\newcommand{\ep}{\scriptsize\mbox{\textcircled{$\dagger$}}}
\newcommand{\core}{\scriptsize\mbox{\textcircled{\#}}}
 \def\cnm{{\mathbb C}^{n\times m}}
\def\cnn{{\mathbb C}^{n\times n}}
\def\cmm{{\mathbb C}^{m\times m}}
\def\cmn{{\mathbb C}^{m\times n}}

\usepackage{sectsty}
\definecolor{astral}{RGB}{46,116,181}
\sectionfont{\color{astral}}
\newtheorem{theorem}{Theorem}[section]
\newtheorem{lemma}[theorem]{Lemma}
\newtheorem{corollary}[theorem]{Corollary}
\newtheorem{proposition}[theorem]{Proposition}
\newtheorem{definition}[theorem]{Definition}
\newtheorem{example}[theorem]{Example}
\newtheorem{remark}[theorem]{Remark}

\newcommand{\D}{{\mathrm D}}

\begin{document}
\begin{frontmatter}
\title{ {\bf
Characterizations of  Weighted Generalized Inverses
}}
\author{Bibekananda Sitha$^{a\dagger}$, ~Ratikanta Behera$^{b}$,~ Jajati Keshari Sahoo$^{c\dagger}$,~ \\
R. N. Mohapatra$^{d}$,~Predrag Stanimirovi\'c$^{e}$}

\address{
 $^{\dagger}$ Department of Mathematics,
BITS Pilani, K.K. Birla Goa Campus, Goa, India
\\\textit{E-mail\,$^a$}: \texttt{p20190066\symbol{'100}goa.bits-pilani.ac.in}\\
\textit{E-mail\,$^c$}: \texttt{jksahoo\symbol{'100}goa.bits-pilani.ac.in}\\
$^{b}$Department of Computational and Data Sciences,
Indian Institute of Science, Bangalore, India.\\
   \textit{E-mail\,$^b$}: \texttt{ratikanta\symbol{'100}iisc.ac.in}\\
   $^{d}${Department of Mathematics,
University of Central Florida, Orlando, Florida, USA\\ \textit{E-mail\,$^d$}: \texttt{ram.mohapatra\symbol{'100}ucf.edu}}\\
$^{e}$ University of Ni\v s, Faculty of Sciences and Mathematics, Ni\v s, Serbia\\
\textit{E-mail}: \texttt{pecko@pmf.ni.ac.rs}\\
}
\begin{abstract}
The main objective of this paper is to introduce unique representations and characterizations for the weighted core inverse of matrices. We also investigate various properties of these inverses and their relationships with other generalized inverses. Proposed representations of the matrix-weighted core inverse will help us to discuss some results associated with the reverse order law for these inverses. Furthermore, this paper introduces an extension of the concepts of generalized bilateral inverse and $\{1,2,3,1^k\}$-inverse and their respective dual for complex rectangular matrices. Furthermore, we establish characterizations of EP-ness and the condition when both $W$-weighted $\{1,2,3\}$ and $W$-weighted $\{1,2,3,1^k\}$ inverses coincide. Then, a W-weighted index-MP, W-weighted MP-index, and W-weighted MP-index-MP matrices for rectangular complex matrices is introduced. In addition, we define the dual inverses for both weighted bilateral inverses and $\{1,2,3,1^k\}$-inverse. Characteristics that lead to self-duality in weighted bilateral inverses are also examined.

\end{abstract}
\end{frontmatter}

\section{Introduction}

\label{sec1}

\subsection{Background and motivation}
Baksalary and Trenkler presented the concept of the core inverse of a square matrix and examined the existence of such matrices in \cite{BakTr10,BakTr14}.
However, the core inverse, initially as termed as the right weak generalized inverse (see \cite{BenIsrael03, Cline68}).
The topic of the core inverse has been extensively explored by many researchers (refer to \cite{ GaoChen18,HaiTing19,WangLi19,zhou2019core}), along with its relations with other generalized inverses \cite{kurata,rakic,wang2015}.
In \cite{ManMo14}, Prasad and Mohana investigated the core-EP inverse for square matrices of arbitrary index.
Following this, the authors in \cite{fer18} extended this concept to study the weighted core-EP inverse and explored various properties associated with these weighted core-EP inverses.
Later on, a thorough discussion on several characterizations of the $W$-weighted core-EP inverses takes place in \cite{gao2018}.
Mosi{\'c} in \cite{mosic2019} studied the weighted core-EP inverse of an operator between the Hilbert spaces.
Moreover, various results have been established for elements of rings with involution \cite{DijanaChu18,rakic2015,ZhangXu17}.
In addition, the authors of \cite{MaH19} characterized the weighted core-EP inverse using matrix decomposition.
The concept of weighted-EP matrices, a generalization of EP matrices, has been explored by Tian and Wang (see \cite{TianWang11}).
In the context of weighted generalized inverses, the importance of the generalized weighted Moore-Penrose inverse is found from its inclusion of {the} weighted Moore-Penrose inverse, the Moore-Penrose inverse, and an ordinary matrix inverse.
In \cite{RaoMitra71} and \cite{khatri1968generalized}, there have been extensive discussions regarding the characterizations and presentations of generalized inverses for matrices.
Moreover, these inverses have wide-ranging applications, with statistics being a significant field where they are utilized (see \cite{RaoRao98,Katri68}).
Similar to the generalized inverses for a commutative ring, the authors of \cite{BapatRao90,PrasadBapa91} introduced and examined a necessary and sufficient criterion for the existence of a generalized inverse, based on the matrix minors over an integral domain context. Within this framework, Prasad and Bapat introduced the concept of the weighted Moore-Penrose inverse in \cite{PrasadBapat92}, wherein they established that the invertibility condition for the existence of the inverse was sufficient. It is worth mentioning that the study of weighted matrix inverses has been investigated in recent studies such as \cite{PredragKM17, PreMo20} and \cite{ZhangWei16}.
The vast literature on weighted inverses and their wide-ranging extensions and applications in different areas of mathematics motivated us to study new representations and characterizations of the weighted core inverse of a matrix.
Recently, Wu and Chen \cite{Wu} defined a new generalized inverse, namely, the $\{1,2,3,1^\kappa\}$-inverse of a square matrix of an arbitrary index. Further, Chowdhry and Roy discussed  W-weighted $\{1,2,3,1^\kappa\}$-inverse for complex rectangular matrices in \cite{chowdhry2023wweighted}. Recently, Ehsan and Abbas \cite{bilateral} presented the notion of the generalized bilateral inverse of a matrix, and a few characterizations of this inverse were investigated. Mosi{\'c} et al. introduced index-MP, MP-index, and MP-index-MP inverse for square matrices by combining the Moore-Penrose inverse with $A^k$, where $k$ is the index of $A$. 

The reverse order law for a generalized inverse serves as a significant factor in theoretical investigation and numerical computations across various domains (see \cite{RaoMitra71,SunWei02,WangYimin18}). For instance, consider two invertible matrices, $A$ and $B$. In such scenarios, the equation $(AB)^{-1}=B^{-1}A^{-1}$ is denoted as the reverse order law. Although this statement is always valid for invertible matrices, but not true in general for generalized inverse \cite{BenIsrael03}. In this regard, Greville \cite{Grev1966} was the first to examine a necessary and sufficient condition for this equivalence within the context of the Moore-Penrose inverse in 1966. Following this, Baskett and Katz \cite{baskett1969} examined the reverse order law for $EP_r$ matrices. Additionally, Deng \cite{deng2011} investigated this concept on the group inverse. From then onwards, many researchers have focused on investigating the reverse order law for different types of generalized inverses (see \cite{MosiDj11,MosiDij12,PaniBeheMi20,JR_rev,YiminWei98}). In light of this context, we discuss several necessary and sufficient conditions for the reverse order law for the weighted core inverse of matrices.

Motivated by the work of \cite{bilateral,mosic2018weighted,indexmp,Wu}, we further characterize weighted core inverse, extend the notion of $\{1,2,3,1^k\}$-inverse, MP-index inverse, index-MP inverse, MP-index-MP inverse, bilateral inverse for rectangular matrices by considering a suitable weight matrix $W$. 

A brief overview of the key points of discussion is as follows.
\begin{enumerate}
\item[(1)] Characterization of $M$-weighted core inverse based on other generalized inverse discussed. The establishment of necessary and sufficient conditions that ensure the validity of the reverse order law for both $M$-weighted core inverse and $N$-weighted dual-core inverse are demonstrated.

\item[(2)] We study different properties and  characterizations for  the $W$-weighted $\{1,2,3,1^k\}$-inverse of rectangular matrices. Some properties of $W$-weighted $\{1,2,3,1^k\}$-inverse with other generalized inverses are presented.

\item[(3)] We introduce a W-weighted index-MP, W-weighted MP-index, and W-weighted MP-index-MP matrices  for rectangular complex matrices. A few characterizations of these inverses have been established.

\item[(4)] We establish weighted bilateral inverse and the dual of weighted generalized bilateral inverses, as well as necessary and sufficient conditions for the self-duality of matrices.
\end{enumerate}


\section{Preliminaries}\label{SecPrelim}

\subsection{Notation and Definitions}
For a matrix $A \in \mathbb{C}^{m \times n}$, let $I_{n}, A^*, \nl(A)$, and $\rg(A)$ denote the identity ${n \times n}$ matrix, conjugate-transpose matrix, null space and range space of $A$, respectively. $P_{L,M}$ denotes a projector onto $L$ along $M$. The index of $A \in \mathbb{C}^{n\times n }$, denoted by $\kappa=\mathrm{ind}(A)$, is the smallest non-negative integer $\kappa$ such that $\mathrm{rank}(A^\kappa)=\mathrm{rank}(A^{\kappa+1})$.
Specifically, when $\kappa=1$, $A$ is known to be of index one matrix, a core matrix, or a group matrix.
To simplify notation, we introduce $A^{(\lambda)}$ to indicate an element of the $\{\lambda \}$-inverse of $A$, and $A\{\lambda\}$ to represent $\{\lambda \}$-inverses of $A$, where $\lambda\subseteq \{1,1^\kappa,2,3,\ldots 3^M,4,4^N,5, \ldots\}$.

Table \ref{TabCEP1} collects matrix equations considered in this research.
The matrices $M$ and $N$ in Table \ref{TabCEP1} are assumed to be positive definite and of appropriate order.
Equations which define weighted generalized inverses are arranged in Table \ref{TabWCEP1}.

\begin{table}[!htp]
 \caption{\small Equations which define generalized inverses.}\label{TabCEP1}
\centering
 {
\tabcolsep 1pt
\begin{tabular}{|c|l|c|l|c|l|c|l|}
 \hline
Label & Equation  & Label & Equation & Label & Equation& Label & Equation\\\hline
$(1)$ & $AZA=A$ & $(2)$ & $ZAZ=Z$ & $(3)$ & $(AZ)^*=AZ$ & $(4)$ & $(ZA)^*=ZA$\\\hline
$(3^M)$ & $(MAZ)^* = MAZ$ & $(4^N)$ & $(NZA)^* = NZA$ & $(5)$ & $AZ=ZA$ &  & \\\hline
$\left(1^\kappa\right)$ &   $ZA^{\kappa+1}=A^k$ & $(6)$ & $ZA^2 = A$ & $(7)$ & $AZ^2 = Z$ & $(9)$ & $Z^2A = Z$ \\\hline
\end{tabular}
}
\end{table}

\begin{table}[!htp]
 \caption{\small Equations which define weighted generalized inverses.}\label{TabWCEP1}
\centering
 {
\tabcolsep 1pt
\begin{tabular}{|c|l|c|l|c|l|c|l|}
 \hline
Label & Equation  & Label & Equation & Label & Equation \\\hline
$(1^W)$ & $AWZWA=A$ & $(2^W)$ & $ZWAWZ=Z$ & $(3^W)$ & $(WAWZ)^*=WAWZ$ \\\hline $(4^W)$ & $(ZWAW)^*=ZWAW$ & $(5^W)$ & $AWZ=ZWA$& $(6^W)$ & $A(WZ)^2=Z$\\\hline
\end{tabular}
}
\end{table}

\smallskip
A summarization of particular composite outer inverses on square matrices is presented in Table \ref{tbl:table1OMP}.

\begin{table}[!htb]
       \caption{Particular cases of composite outer inverses on square matrices.}\label{tbl:table1OMP}
    \begin{tabular}{p{3.5cm} @{\hspace{0.4cm}} p{3cm} @{\hspace{0.5cm}} p{6cm} @{\hspace{0.5cm}} p{2cm} } \hline\noalign{\vskip 0.2cm}
		\textbf{Restrictions} & \textbf{Title} & \textbf{Composite outer inverse} & \textbf{Reference} \\ \hline	
            $\mathrm{ind}(A)=1$  & core  & $ A^{\tiny\textcircled{\#}}=A^{\#}AA^{\dag}$  &  \cite{BakTr10} \\
            $\mathrm{ind}(A)=1$ & dual core  &  $ A_{\tiny\textcircled{\#}}=A^{\dag} AA^{\#}$  &  \cite{BakTr10} \\
            $\mathrm{ind}(A)=\kappa$ & DMP &  $ A^{\D,\dag}=A^\D AA^{\dag}$  &  \cite{malik} \\
            $\mathrm{ind}(A)=\kappa$ & MPD &  $ A^{\dag, \D}=A^{\dag} AA^\D$  &  \cite{malik} \\
             $\mathrm{ind}(A)=\kappa$  & CMP &  $ A^{c,\dag}=A^\dag AA^\D AA^\dag$  &  \cite{Mehdi} \\
 $\mathrm{ind}(A)=\kappa$  & $\kappa$-MP &  $ A^{k,\dag}=A^kAA^\dag$  &  \cite{indexmp}\\
 $\mathrm{ind}(A)=\kappa$  & MP-$\kappa$ &  $ A^{k,\dag}=A^{\dag} AA^k$  &  \cite{indexmp}\\
 $\mathrm{ind}(A)=\kappa$  & MP-$\kappa$-MP &  $ A^{\dag,k,\dag}=A^\dag AA^kAA^{\dag}$  &  \cite{indexmp}\\ 
            \hline
        \end{tabular}
\end{table}

Composite outer inverses are surveyed in Table \ref{tableOMP}.

\begin{table}[!htb]\caption{\small{Survey of W-weighted composite outer inverses.}}\label{tableOMP}
	\centering
	\begin{tabular}{p{2cm} @{\hspace{0.6cm}} p{5cm} @{\hspace{5cm}} p{3.5cm} } \hline\noalign{\vskip 0.2cm}
		\textbf{Title} & \textbf{Definition} & \textbf{Reference} \\ \hline
		WDMP & $A^{\D,\dagger,W}=WA^{\D,W}WAA^{\dagger}$ & \cite{DMPRectangular} \\
        WMPD & $A^{\dagger,\D,W}=A^{\dagger}AWA^{\D,W}W$ & \cite{w-mpd}\\
        WCMP & $A^{c,\dagger,W}=A^{\dagger}AWA^{\D,W}WAA^{\dagger}$ & \cite{Mosicw-cmp}\\
\hline
	\end{tabular}
\end{table}

%
\begin{definition}
    For $A\in \mathbb{C}^{n \times n}$ with $\mathrm{ind}(A)=\kappa$, a matrix $X$ is termed as:
    \begin{enumerate}
        \item[\rm(a)] The Drazin inverse $A^{\D}$ of $A$ if  $X\in A\{1^\kappa,2,5\}$ {\rm \cite{draz}};
        \item[\rm(b)] core-EP inverse $A^{\ep}$ of $A$ if $X\in A\{1^\kappa,2,7\}$ {\rm \cite{ManMo14}}.
    \end{enumerate}

\end{definition}
In the above definition, if $\mathrm{ind}(A) = 1$, then the Drazin inverse becomes the group inverse $A^{\#}$ and the core-EP inverse becomes core inverse $A^{\core}$.

Next we recall the following weighted generalized inverses.

\begin{definition}\label{wmpi}
Let $A\in\cmn$, $M\in \cmm$ and $N\in \cnn$ be two invertible hermitian matrices. A matrix $X\in\cnm$ is called
the { weighted Moore-Penrose inverse} of $A$ if  $X\in A\{1,2,3^{M},4^{N}\}$ and it is denoted by $A^{\dagger}_{M,N}$ {\rm \cite{RaoMitra71}}.
\end{definition}

\begin{remark}
    Throughout this paper, we fix $W\in \mathbb{C}^{n \times m}$ as a nonzero weight matrix, unless explicitly mentioned.
\end{remark}
We now discuss the weighted generalized inverses for rectangular matrix.
\begin{definition}
    For $A\in \mathbb{C}^{m \times n}$ with $\max\{\mathrm{ind}(AW)$, $\mathrm{ind}(WA)\}=\kappa$, then a unique matrix $X$ is:
    \begin{enumerate}
        \item[\rm(a)] W-weighted Drazin inverse $A^{\D,W}$ if  $X\in A\{2^W,5^W\}$ and $XW(AW)^{\kappa+1}=(AW)^\kappa$ {\rm \cite{RectangularDrazin}}.
        \item[\rm(b)] W-weighted core-EP inverse $A^{\ep,W}$ if $WAWX=(WA)^k[(WA)^\kappa]^{\dagger}$ and $\rg(X) \subseteq \rg((AW)^\kappa)$ {\rm \cite{fer18}}.
    \end{enumerate}
\end{definition}
The identities $A^{\D,W}W=(AW)^{\D}$ and $WA^{\D,W}=(WA)^{\D}$ hold \cite{RectangularDrazin}.
Furthermore, the following properties and representations are also related to $A^{\D,W}$:
$$A^{\D,W}=A((WA)^{\D})^2=((AW)^{\D})^2A;$$
$$\mathrm{rank}(AW)^\kappa=\mathrm{rank}(WA)^\kappa;$$
$$\rg(A^{\D,W})=\rg(AW)^\kappa \mbox{ and }\nl(A^{\D,W})=\nl(WA)^\kappa. $$
To explore more on the weighted Drazin inverse, we refer to \cite{Hernazwdrazinpreorder,integralDrazin,peturbWdrazin}.

 Following this, the authors of \cite{gao2018} have investigated the $W$-weighted core inverse further, presenting a new expression in Theorem 2.2, which reads as follows:

\begin{definition}
 Let $A \in \mathbb{C}^{m \times n}$  and  $\kappa=\max\{\mathrm{ind}(AW),\mathrm{ind}(WA)\}$.
 A matrix $X$ is called the $W$-weighted core-EP inverse if $X\in A\{3^W,6^W\}$ and $XW(AW)^{\kappa+1}=(AW)^\kappa$.
\end{definition}
For $\kappa=1$, it corresponds to the weighted core inverse of $A$, denoted by $A^{\core,W}$.
 {The authors of \cite{gao2018} proved that the  $W$-weighted core-EP inverse of $A$ is uniquely defined by
\begin{equation*}
    A^{\ep,W}=A[(WA)^{\ep}]^2,\ \ ~A^{\ep,W}=A^{\D,W}P_{(WA)^\kappa}.
\end{equation*}
 The concept for the  following definition is drawn from \cite{TianWang11} (see Theorem 3.5).
\begin{definition}
 Let $M,~N\in \mathbb{C}^{n \times n}$ be two hermitian invertible matrices and $A\in \mathbb{C}^{n \times n}$ be a core matrix.
 Then $A$ is called weighted-EP with respect to $(M,N)$ if $A^{\dagger}_{M,N}$ exists and
$A^{\dagger}_{M,N}=A^{\core}$.
\end{definition}

 One can observe that, the core-EP inverse is unique and  $A^{\core}\in A\{1,2\}$.
 Using the definition above, the following result can be easily demonstrated:
\begin{proposition}\label{pro2.2}
For $A \in \mathbb{C}^{n \times n}$, if $X\in A\{6,7\}$, then $X\in A\{1,2\}$.
\end{proposition}

Inspired by the core inverse of matrices, Sahoo {\it et al}. recently extended this notion to tensors, introducing the core inverse and the core-EP inverse using the Einstein product, as detailed in \cite{JR_rev} and \cite{SahBe20}.
Let us recall the following result from \cite{RaoMitra71}.

\begin{lemma}\label{lem1.2}
Let $A\in\cnn$.
If $A=A^2X=YA^2$ is satisfied for some $X,~Y\in\cnn$  then $A^{\#}$ exists and $A^{\#}=AX^2=YAX=Y^2A$.
\end{lemma}

Next, we state the rank cancellation rule which was introduced by Matsaglia and  Styan \cite{mat}. See also Rao and Bhimasankaram \cite[Theorem 3.5.7]{raob}.

\begin{lemma}\label{rtcan}
Let $A,B,C$ and $D$ be the matrices of proper dimensions.
If $CAB = DAB$ and $\mathrm{rank}(AB) = \mathrm{rank}(A)$, then $CA = DA$.
\end{lemma}

Next, we give the definition of $W$-weighted $\{1,2,3,1^k\}$-inverse of rectangular matrix based upon the concept of the $\{1,2,3,1^k\}$-inverse of a square matrix of arbitrary index introduced by Wu and Chen \cite{Wu}.
\begin{definition}\label{def3,1k}
The matrix $X \in \mathbb{C}^{m \times n}$ is called $W$-weighted $\{1,2,3,1^k\}$-inverse of $A\in \mathbb{C}^{m \times n}$ if
\begin{equation}
\begin{alignedat}{2}
    (1)~ WAWXWAW&=WAW, &(2)~XWAWX &=X\\
    (3)~(WAWX)^*&=WAWX,\
    &~~(1^{k})~XW(AW)^{k+1}&=(AW)^k \mbox{ for some integer}~k\geqslant 0.
\end{alignedat}
\end{equation}
\end{definition}
An arbitrary $\{1,2,3,1^k\}^W$-inverse of $A$ will be denoted by $A^{(1,2,3,1^k)^W}$ and the set of all $A^{(1,2,3,1^k)^W}$ inverses by $A\{1,2,3,1^k\}^W$. 
The following example illustrates the above definition with the aim to show that the $W$-weighted $\{1,2,3,1^k\}$-inverse of $A$ is different than other generalized inverses.
\begin{example}
For $A=\begin{bmatrix}
1 & i\\
0 & 0\\
0 & 0
\end{bmatrix}$, we can find $W=\begin{bmatrix}
1 & 1 & 0\\
0 & 0 & 0
\end{bmatrix}$
and $A^{W-1,2,3,1^k}=\begin{bmatrix}
1 & 0\\
0 & 0\\
0 & 0
\end{bmatrix}$ for any integer $k\geqslant 0$. But $A^{\D,W}=\begin{bmatrix}
1 & i\\
0 & 0\\
0 & 0
\end{bmatrix},~A^{\dagger}=\begin{bmatrix}
\frac{1}{2} & 0 & 0\\
-\frac{i}{2} & 0 & 0
\end{bmatrix}=A^{c,\dagger,W}$, and $A^{\D,\dagger,W}=\begin{bmatrix}
1 & 0 & 0\\
0 & 0 & 0
\end{bmatrix}$.

Hence, $A^{(1,2,3,1^k)^W}\neq A^{\D,W} \neq A^{\D,\dagger,W} \neq A^{c,\dagger,W}.$
\end{example}
\begin{definition}\label{def4,k1}
Let $A\in \mathbb{C}^{m \times n}$. Then $X \in \mathbb{C}^{m \times n}$ is called $W$-weighted $\{1,2,4,{}^k1\}$-inverse of $A$ if
\begin{equation}
\begin{alignedat}{2}
    (1)~ WAWXWAW&=WAW, &(2)~XWAWX &=X\\
    (4)~(XWAW)^*&=XWAW,
    &~~~(^k1)~(WA)^{k+1}WX&=(WA)^k \mbox{ for some integer}~k\geqslant 0.
\end{alignedat}
\end{equation}
\end{definition}
In particular settings $m=n$ and $W=I_{n}$ Definition \ref{def3,1k} and Definition \ref{def4,k1} becomes Definition 1.2 in \cite{Wu}.

In order to illustrate some properties of the $W$-weighted $\{1,2,3,1^k\}$-inverse of a matrix $A\in \mathbb{C}^{m \times n}$ , we need to present the following known results.
First we generalize characterizations of the $W$-weighted core inverse from the weighted core-EP inverse.
Now, using \cite{fer18,gao2018,rakic}, we give some general characterizations of the W-weighted core inverse.
\begin{lemma}\label{lem1.5}
The following statements are equivalent for $A,X \in \mathbb{C}^{m \times n}$:
\begin{enumerate}
    \item[\rm (i)] $X$ is W-weighted core inverse of $A$;
    \item[\rm (ii)] $WAWXWAW=WAW,~XWAWX=X,~(WAWX)^*=WAWX,~XW(AW)^2=AW,~\\
          A(WX)^2=X$;
    \item[\rm (iii)] $WAWXWAW=WAW,~(WAWX)^*=WAWX,~A(WX)^2=X$;
    \item[\rm (iv)] $(WAWX)^*=WAWX,~XWAWX=X,~XW(AW)^2=AW$;
    \item[\rm (v)] $(WAWX)^*=WAWX,~XW(AW)^2=AW,~A(WX)^2=X$.
\end{enumerate}
\end{lemma}

\begin{lemma} \label{lemma1.6} {\rm \cite{Wu}}
Let $A,B\in \mathbb{C}^{m\times m}$. If $AZB=0$ for all $Z\in \mathbb{C}^{m\times m}$ then either $A=0$ or $B=0$.
\end{lemma}

%
%

 For a nonzero weight $W\in \mathbb{C}^{n\times m}$, the weighted inner, outer, and reflexive inverses of $A\in \mathbb{C}^{m\times n}$ are denoted as the sets
 $A\{1^W\}$, $A\{2^W\}$ and $A\{1^W,2^W\}$, respectively, and defined as follows:

 \begin{align*}
    A\{1^W\}&:=\{X\in \mathbb{C}^{m \times n}|\ AWXWA=A\};\\
    A\{2^W\}&:=\{X\in \mathbb{C}^{m \times n}|\ XWAWX=X\};\\
    A\{1^W,2^W\}&=G_{rw}(A):=A\{1^W\}\cap A\{2^W\}.
\end{align*}

\begin{definition}{\rm \cite{bilateral}}
    Let $A\in \mathbb{C}^{m\times n}$ and $X_1,X_2\in A\{1\}\cup A\{2\}$. Then the generalized bilateral inverse of $A$ based on $X_1$ and $X_2$ is defined as
    $$
    A^{ X_1\looparrowright X_2} =X_1AX_{2}.
    $$
\end{definition}
The dual of generalized bilateral inverse of $A$ is defined as
$$A^{ X_2\looparrowright X_1}=X_2AX_{1}.$$

\section{Further properties of weighted core inverse}
 We recall $M$-weighted core and $N$-weighted dual core inverse for matrices, which was first introduced by Mosic {\it et al}. in \cite{DijanaChu18} for elements in a ring with involution.
 Following that, we discuss a few new representations and  characterizations of these inverses.

\begin{definition}{\rm \cite[Theorem 2]{DijanaChu18}}\label{mcore}
Let $M\in \cnn$ be an invertible hermitian matrix and $A \in \mathbb{C}^{n \times n}$. A matrix $X\in\cnn$ satisfying
$$\left(3^M\right)~(MAX)^* = MAX,~~\ (6)~XA^2 = A,~~\ (7)~\ AX^2 = X,
$$
is called the $M$-weighted core inverse of $A$ and is denoted by $A^{\core,M}$.
\end{definition}


Using Proposition \ref{pro2.2}, we give a new proof for the uniqueness of the $M$-weighted core inverse given in the next theorem.

\begin{theorem}\label{unimcore}
Let $M\in \cnn$ be an invertible  hermitian matrix and $A \in \mathbb{C}^{n \times n}$. Then the $M$-weighted core inverse of $A$ is unique in the case of its existence.
\end{theorem}
\begin{proof}
Suppose there exist two $M$-weighted core inverses of $A$, say $X$ and $Y$.
Results of Proposition \ref{pro2.2} lead to the concussion
\begin{equation*}
\aligned
Y&=YAY=YAXAY\\
&=YAXM^{-1}MAY=YAXM^{-1}(MAY)^*\\
&=YAXM^{-1}Y^*A^*M=YM^{-1}MAXM^{-1}Y^*A^*M\\
&=YM^{-1}X^*A^*Y^*A^*M=YM^{-1}X^*A^*M=YAX\\
&=YA^2X^2=AX^2\\
&=X,
\endaligned
\end{equation*}
which proves the uniqueness.
\end{proof}

\begin{definition}\label{ncore}
Let $N\in \cnn$ be an invertible hermitian matrix and $A \in \mathbb{C}^{n \times n}$. Then a matrix $X\in\cnn$ satisfying
$$
\left(4^N\right)~(NXA)^* = NXA,~~\ (8)~\ A^2X = A,~~\ (9)~\ X^2A = X,
$$
is called {the} $N$-weighted dual core inverse of $A$ and is denoted by $A^{N,\core}$.
\end{definition}

\begin{example}\rm
For $A= \begin{bmatrix}
1 & 1\\
0 & 0
\end{bmatrix}$, $N=\begin{bmatrix}
2 & i\\
-i & 2
\end{bmatrix}$
simple verification confirms that the matrix\\
$Y= \begin{bmatrix}
0.5-0.25i & 0.5-0.25i\\
0.5+0.25i & 0.5+0.25i
\end{bmatrix}$
satisfies the equations $\left(4^N\right)$, $(8)$ and $(9)$ in Definition \ref{ncore}.
\end{example}

The following results are valid as analogies with Proposition \ref{pro2.2} and Theorem \ref{unimcore}.
\begin{proposition}\label{pro2.5}
{For $A \in \mathbb{C}^{n \times n}$, if} the matrix $X\in A\{8,9\}$, then $X\in A\{1,2\}$.
\end{proposition}
\begin{theorem}\label{unincore}
Let $N\in \cnn$ be an invertible hermitian matrix and $A \in \cnn$. Then the $N$-weighted dual core inverse of $A$ is unique (if it exists).
\end{theorem}

Next we recall the equivalency between the group inverse and $M$-weighted core inverse.

\begin{lemma}{\rm \cite[Theorem 2]{DijanaChu18}}\label{thm2.7}
Let $M \in \cnn$ be an invertible hermitian matrix and $A\in \cnn$.
Then the following statements are equivalent:
\begin{enumerate}
\item[\rm (i)] $(MAX)^*=MAX$, $XA^2=A$, and $AX^2=X$;
\item[\rm (ii)] $A^{\#}$ exists and $X\in A\{1,3^M\}$.
\end{enumerate}
  \end{lemma}

The subsequent corollaries provide an equivalent definition for the $M$-weighted core inverse.
\begin{corollary}\label{cor2.10}
Let  $M \in \cnn$ be an invertible hermitian matrix and $A\in\cnn$.   If $X\in\cnn$ satisfies $XA^2=A$ and $(MAX)^*=MAX$ then $A^{\core,M}=X$.
\end{corollary}

\begin{proof}
Let  $XA^2=A$. Then     $AXA=AXAA^{\#}A=AXA^2A^{\#}=A^2A^{\#}=A$.
Thus $X\in A\{1,3^M\}$ and hence by the Lemma \ref{thm2.7}, $A^{\core,M}=X=A^{\#}AX$.
\end{proof}
\begin{remark}
 If we replace the condition $XA^2=A$ in the Corollary \ref{cor2.10}, by $XA=A^{\#}A$ then the corresponding statement is still true since   $XA=A^{\#}A\Longleftrightarrow XA^2=A$.
\end{remark}

\begin{corollary}\label{prop12}
Let $M \in \cnn$ be an invertible hermitian matrix and  $A\in \cnn$. If a matrix $X$ satisfies
\begin{center}
    $(1)$\  $ {A}  {X}  {A}= {A}$,\ $(3^M)$\  $(M {A}  {X})^*= M{A}  {X}$ and $(7)$\ $ {A}  {X}^2= {X},$
\end{center}
then $ {X}$ is the $M$-weighted core inverse of $ {A}.$
\end{corollary}
\begin{proof}
From  $ {A}  {X}  {A}= {A}$ and $ {A}  {X}^2= {X}$, we get  $ {A}= {A}^2  {X}^2  {A}.$ Thus $\rg(A)\subseteq \rg(A^2)$ and subsequently $A$ is group invertible. Hence by Lemma \ref{thm2.7}, ${X}$ is identified as the $M$-weighted core inverse of ${A}$.
\end{proof}

\begin{theorem}\label{Algotheorem}
Let  $M \in \cnn$ be an invertible hermitian matrix and $A\in\cnn$  {with $\mathrm{ind}(A)=1$}.
Then $A^{\core,M}$ exists 
and
\begin{equation}\label{EquMEP1}
A^{\core,M}=A(A^*M A^2)^{(1)}A^*M=A^{(1,2)}_{\mathcal{R}(A),\mathcal{N}(A^*M)}.
\end{equation}
\end{theorem}

\begin{proof}
Since $M$ is invertible it follows $\mathrm{rank}(A^*M)=\mathrm{rank}(A^*MA)$, which further implies
\begin{equation}\label{eq32}
  A^*M=A^*MAU \mbox{ for some }U\in\cnn.
\end{equation}
Let $X=A(A^*MA^2)^{(1)}A^*M$.
Then application of equation \eqref{eq32} initiates
\begin{eqnarray*}
AXA^2&=&A^2(A^*MA^2)^{(1)}A^*MA^2=M^{-1}(MA)A(A^*MA^2)^{(1)}A^*MA^2\\
&=&M^{-1}U^*A^*MA^2(A^*MA^2)^{(1)}A^*MA^2=M^{-1}U^*A^*MA^2\\\
&=&M^{-1}MA^2=A^2.
\end{eqnarray*}
Thus by applying Lemma \ref{rtcan} (rank cancellation law), we obtain $AXA=A$. Now
\begin{eqnarray*}
MAX&=&MA^2A^*MA^2)^{(1)}A^*M=MA^2A^*MA^2)^{(1)}A^*MAU\\
&=&U^*A^*MA^2(A^*MA^2)^{(1)}A^*MA^2A^{\#}U=U^*A^*MA^2A^{\#}U\\
&=&U^*A^*MAU.
\end{eqnarray*}
Thus $(MAX)^*=MAX$.

Hence by Lemma \ref{thm2.7}, $A^{\core,M}=X=A(A^*MA^2)^{(1)}A^*M$.

\smallskip
Finally, the representation $A^{\core,M}=X=A(A^*MA^2)^{(1)}A^*M$ is obtained using Urquhart result \cite{Urquhart} and \cite[Corollary 1, P. 52.]{BenIsrael03}.
\end{proof}

Following Theorem \ref{Algotheorem}, the computation of ${A}^{\core,M}$ is explained in Algorithm \ref{AlgrshInInv}.
Here ${\rm rref}[A|\ I_n]$ denotes the row reduced echelon form of $[A|\ I_n]$.

\begin{algorithm}[H]
	\begin{algorithmic}[1]\caption{Computation of ${A}^{\core,M}$}\label{AlgrshInInv}
		\smallskip
		\REQUIRE The matrices $M,A\in \mathbb C^{n\times n}$.
\STATE Compute $B=A^*MA^2.$
        \STATE $r\leftarrow \mathrm{rank}(B)$.
		
		\STATE $C\leftarrow \mathrm{rref}[B, I_{n}]$.
		\STATE $E\leftarrow $ last $n$ columns of $B$.
		
		\STATE Find a permutation matrix $P$ such that $EBP = \bmatrix I_r & K \\ O & O\endbmatrix.$
	    \STATE Compute $Q=P\bmatrix I_r & O \\ O & L\endbmatrix E=(A^*MA^2)^{(1)}$, where $L\in \mathbb C^{(n-r)\times ({n}-r)}$ is an arbitrary matrix.

		\STATE Compute the output
		$
		{A}^{\core,M}:=AQA^*M$
		\end{algorithmic}
\end{algorithm}

\subsection{Reverse order law for weighted core inverse}
The study of the reverse-order law for matrices has been a significant research area in the theory of generalized inverses, as seen in the works of \cite{baskett1969,BenIsrael03,MosiDj11,MosiDij12,JR_rev,YiminWei98}.
In this section, we outline a few necessary and sufficient conditions for the validity of the reverse order law, focusing specifically on the weighted core inverse.

\begin{theorem}\label{thm4.3}
Let $M\in\cnn$ be an invertible hermitan matrix and $ A,B \in\cnn$. If
\begin{equation*}
(AB)^{\core,M}=B^{\core,M}A^{\core,M}
\end{equation*}
then the following two statements are true:
\begin{enumerate}
    \item[\rm (i)] $\rg\left(B^{\core,M}A\right)\subseteq\rg(AB) \subseteq \rg(BA)$;
    \item[\rm (ii)] $BB^{\core, M}A^{\core, M}\in  C\{3^M,6\}$, where $C= ABB^{\core, M}$.
\end{enumerate}
\end{theorem}

\begin{proof}
(i) Let $(AB)^{\core,M}=B^{\core,M}A^{\core,M}$. Then
\begin{eqnarray*}
AB&=& (AB)^{\core,M}(AB)^2=B^{\core,M}A^{\core,M}(AB)^2=B(B^{\core,M})^2A^{\core,M}(AB)^2\\
&=& BB^{\core,M}(B^{\core,M}A^{\core,M})(AB)^2= BB^{\core,M}(AB)^{\core,M}(AB)^2=BB^{\core,M}AB\\
&=& BB^{\core,M}A^{\core,M}A^2B= B(AB)^{\core,M}A^2B=BAB((AB)^{\core,M})^2A^2B.
\end{eqnarray*}
Thus $\rg(AB) \subseteq \rg(BA)$.
Further, the identities
$$B^{\core,M}A=B^{\core,M}A^{\core,M}A^2=(AB)^{\core,M}A^2=AB((AB)^{\core,M})^2A^2,
$$
lead to
$\rg\left(B^{\core,M}A\right)\subseteq\rg(AB)$.

\smallskip
\noindent (ii) The statement $ {B}{B}^{\core, M}  {A}^{\core, M}\in {C}\{3^M\}$ follows from
 \begin{equation*}
 \aligned
 (MCBB^{\core, M}A^{\core, M})^*&= (MABB^{\core, M}  {B}  {B}^{\core, M}  {A}^{\core, M})^*= (M{A}  {B}  {B}^{\core, M}  {A}^{\core, M})^*\\
 &= (M{A}  {B} ( {A}  {B})^{\core, M})^*=M{A}  {B} ( {A}  {B})^{\core, M}\\
 &= M{C}  {B}  {B}^{\core, M}  {A}^{\core, M}.
 \endaligned
 \end{equation*}

 Using $ {A} {B}= {B}  {B}^{\core, M} ( {A}  {B})$ from part $(a)$, we obtain
 \begin{equation*}
 \aligned
  {B} {B}^{\core,M} {A}^{\core,M} {C}^2&=  {B} {B}^{\core,M} {A}^{\core,M} {A} {B} {B}^{\core,M} {A}^{\core,M} {A}^2 {B} {B}^{\core,M}\\
 &= {B} {B}^{\core,M} {A}^{\core,M} {A}^2 {B} {B}^{\core,M}= {B} {B}^{\core,M} {A} {B} {B}^{\core,M}\\
 &= {A} {B} {B}^{\core,M}.
 \endaligned
 \end{equation*}
 Hence $ {B} {B}^{\core,M} {A}^{\core,M}\in {C}\{6\}$ and which completes the proof.
 \end{proof}

Next we present a characterization of the reverse-order law for {the} $M$-weighted core inverse.

\begin{theorem}
Let $M\in\cnn$ be an invertible hermitian matrix and $ A,B \in\cnn$ with $\mathrm{ind}(A)=1=\mathrm{ind}(B)$.
If $A^2=BA$, $A\{1,3^M\} \neq $  {$\emptyset$} and $B\{1,3^M\} \neq$  {$\emptyset$}, then $\mathrm{ind}(AB)=1$ and $(AB)^{\core,M}=B^{\core,M}A^{\core,M}$.
\end{theorem}

\begin{proof}
A straight conclusion is $\rg\left((AB)^2\right)\subseteq \rg(AB)$.
Further,
$$AB=A^5(A^{\#})^4B=A(BA)^2(A^{\#})^4B=(AB)^2(A^{\#})^3B$$
implies $\rg(AB)\subseteq \rg\left((AB)^2\right)$.
This fact leads to $\rg(AB)= \rg\left((AB)^2\right)$ and hence $\mathrm{ind}(AB)=1$.
Repetitive usage of $A^2=BA$ produces
\begin{equation}\label{eqn4.1}
\aligned
A&=A^2A^{\#}=BAA^{\#}=B^{\core,M}B^2AA^{\#}=B^{\core,M}BA^2A^{\#}\\
&=B^{\core,M}A^3A^{\#}=B^{\core,M}A^2,
\endaligned
\end{equation}
and
\begin{equation*}
    AA^{\core,M}=A^2(A^{\core,M})^2=BA(A^{\core,M})^2=BA^{\core,M}.
\end{equation*}
In view of the property \eqref{eqn4.1} it can be concluded
\begin{equation}\label{eqn4.2}
\aligned
    A^{\core,M}&=A(A^{\core,M})^2=B^{\core,M}A^2(A^{\core,M})^2\\
    &=B^{\core,M}BA(A^{\core,M})^2=B^{\core,M}BA^{\core,M},
    \endaligned
\end{equation}
 and
\begin{equation}
\aligned
B^{\core,M}A^{\core,M}(AB)^2&=B^{\core,M}A^{\core,M}A(BA)B\\
&=B^{\core,M}A^{\core,M}A^3B=B^{\core,M}A^2B\\
&=AB.
\endaligned
\end{equation}
From Equation \eqref{eqn4.2}, we have
\begin{equation*}
\aligned
   (MABB^{\core,M}A^{\core,M})^*&=(MABB^{\core,M}B^{\core,M}BA^{\core,M})^*=(MAB^{\core,M}BA^{\core,M})^*\\
   &=(MAA^{\core,M})^*=MAA^{\core,M}=MABB^{\core,M}A^{\core,M}.
   \endaligned
\end{equation*}
Further, from \eqref{eqn4.1} and \eqref{eqn4.2}, we obtain
\begin{equation*}\label{eqn4.5}
\aligned
AB(B^{\core,M}A^{\core,M})^2&=ABB^{\core,M}A^{\core,M}B^{\core,M}A^2(A^{\core,M})^3=ABB^{\core,M}A^{\core,M}A(A^{\core,M})^3\\
&=ABB^{\core,M}(A^{\core,M})^3=ABB^{\core,M}B^{\core,M}BA^{\core,M}(A^{\core,M})^2\\
&=AB^{\core,M}BA^{\core,M}(A^{\core,M})^2=A(A^{\core,M})^3\\
&=B^{\core,M}A^2(A^{\core,M})^3=B^{\core,M}A^{\core,M}.
\endaligned
\end{equation*}
Thus $B^{\core,M}A^{\core,M}$ is the $M$-weighted core inverse of $AB$.
\end{proof}

Similarly, the following theorem can be proved.
\begin{theorem}
Let $N\in\cnn$ be an invertible hermitan matrix and $ A,B \in\cnn$ satisfy conditions $\mathrm{ind}(A)=1=ind(B)$.
If $B^2=BA$, {$A\{1,4^N\} \neq \emptyset$ and $B\{1,4^N\} \neq \emptyset$}, then $\mathrm{ind} (AB)=1$ and $(AB)^{N,\core}=B^{N,\core}A^{N,\core}$.
\end{theorem}

\begin{theorem}
Let $M\in\cnn$ be an invertible hermitan matrix and $ A,B \in\cnn$ with $\mathrm{ind}(A)=1=\mathrm{ind}(B)=\mathrm{ind}(AB)$ and $\rg(A^*MB)=\rg(MBA^*)$.
Then
\begin{equation*}
(AB)^{\core,M}=B^{\core,M}A^{\core,M}
\end{equation*}
 if and only if one of the subsequent conditions holds:
\begin{enumerate}
    \item[\rm (i)] $\rg(B^{\core,M}A) \subseteq \rg(AB) \subseteq \rg(BA)$ and $MBB^{\core,M}AA^{\core,M}=MAA^{\core,M}BB^{\core,M}$.
    \item[\rm (ii)]   $\rg(B^{\core,M}A) \subseteq \rg(AB) \subseteq \rg(BA)$ and  $BB^{\core,M}AA^{\core,M}=AA^{\core,M}BB^{\core,M}$.
\end{enumerate}
\end{theorem}

\begin{proof}
Let $(AB)^{\core,M}=B^{\core,M}A^{\core,M}$. Then part (i) is follows from Theorem \ref{thm4.3}. Further,
\begin{equation}\label{eq4.5}
    AB=B^{\core,M}A^{\core,M}(AB)^2=B(B^{\core,M})^2A^{\core,M}(AB)^2=BB^{\core,M}AB.
\end{equation}
Using the relation of $\rg(A^*MB)=\rg(MBA^*)$, we obtain $B^*M^*A=UAB^*M^*$ for some $U\in \cnn$.
In addition,
\begin{eqnarray}\label{eq4.6}
\nonumber
B^*M^*A&=&UA(MB)^*=UA(MBB^{\core,M}B)^*=UAB^*(MBB^{\core,M})^*= UAB^*MBB^{\core,M}\\
&=&B^*M^*ABB^{\core,M}.
\end{eqnarray}
Using $\rg(AB) \subseteq \rg(BA)$, equation\eqref{eq4.5}, and equation \eqref{eq4.6}, we obtain
\begin{equation}\label{eq4.60}
\aligned
MBB^{\core,M}AA^{\core,M}&= (MBB^{\core,M})^*AA^{\core,M}= (B^{\core,M})^*B^*M^*AA^{\core,M}\\
&= (B^{\core,M})^*B^*M^*ABB^{\core,M}A^{\core,M}=(MBB^{\core,M})^*ABB^{\core,M}A^{\core,M}\\
&= MBB^{\core,M}ABB^{\core,M}A^{\core,M}=MABB^{\core,M}A^{\core,M}\\
&= (MABB^{\core,M}A^{\core,M})^*=(MBB^{\core,M}M^{-1}MAA^{\core,M})^*\\
&= (MAA^{\core,M})^*M^{-1}(MBB^{\core,M})^*\\
&=MAA^{\core,M}BB^{\core,M}.
\endaligned
\end{equation}

 Conversely, let $\rg({A}{B})\subseteq\rg({B}{A})$. This yields ${A}{B}={B}{A}{U}$ for some ${U} \in \mathbb{C}^{n\times n}$ and
 \begin{equation}\label{eqn4.7}
     {A}{B}={B}{A}{U}={B}{B}^{\core,M}{B}{A}{U}={B}{B}^{\core,M}{A}{B}.
 \end{equation}
 Using equation \eqref{eqn4.7} along with $\rg({A}^*M{B})=\rg(M{B}{A}^*)$, we can easily prove (like equation \eqref{eq4.60})
 \begin{equation}\label{eq4.8}
   M{B}{B}^{\core,M}{A}{A}^{\core,M}=M{A}{B}{B}^{\core,M}{A}^{\core,M} \mbox{ or }  {B}{B}^{\core,M}{A}{A}^{\core,M}={A}{B}{B}^{\core,M}{A}^{\core,M}.
 \end{equation}
 A combination od $MBB^{\core,M}AA^{\core,M}=MAA^{\core,M}BB^{\core,M}$ and \eqref{eq4.8} gives
 \begin{center}
$(M{A}{B}{B}^{\core,M}{A}^{\core,M})^*=(M{B}{B}^{\core,M}{A}{A}^{\core,M})^*=MA{A}^{\core,M}{B}{B}^{\core,M}=MA{B}{B}^{\core,M}A^{\core,M}$ .
 \end{center}
 From $\rg({B}^{\core, M}{A})\subseteq\rg({A}{B})$ and second part of \eqref{eq4.8}, we conclude
\begin{eqnarray*}
{A}{B}({B}^{\core,M}{A}^{\core,M})^2&=&{B}{B}^{\core,M}{A}{A}^{\core,M}{B}^{\core,M}{A}^{\core,M}={A}{A}^{\core,M}{B}{B}^{\core,M}{B}^{\core,M}{A}^{\core,M}\\
&=&{A}{A}^{\core,M}({B}^{\core,M}{A})({A}^{\core,M})^2={A}{A}^{\core,M}({A}{B}{V})({A}^{\core,M})^2\\
&=&({A}{B}{V})({A}^{\core,M})^2=({B}^{\core,M}{A})({A}^{\core,M})^2={B}^{\core,M}{A}^{\core,M}.
\end{eqnarray*}
Further,  ${A}{B}={B}{B}^{\core,M}{A}{B}={B}{B}^{\core,M}{A}{A}^{\core,M}{A}{B}={A}{B}{B}^{\core,M}{A}^{\core,M}{A}{B}$. Thus by  Corollary \ref{prop12}, ${B}^{\core,M}{A}^{\core,M}$ is the $M$-weighted core inverse of $AB$.
\end{proof}

\subsection{$W$-weighted $\{1,2,3,1^k\}$-inverse}
First, we discuss the canonical representations of $W$-weighted $\{1,2,3,1^k\}$-inverse of any given matrix. In \cite{fer18}, it has been proved that for any matrix $A\in \mathbb{C}^{m \times n}$ and $W\in \mathbb{C}^{n \times m}$ with $\kappa=\max\{\text{ind}(AW),\text{ind}(WA)\}$, there exist two unitary matrices $U\in \mathbb{C}^{m \times m},~V\in \mathbb{C}^{n \times n}$ such that the decomposition of $A$ and $W$ for suitable sizes be written as
\begin{equation}\label{Wcep-decompeqn}
        A=U\begin{pmatrix}
        A_{1} &A_{2}\\
        0 &A_{3}
    \end{pmatrix}V^{*}\mbox{ and }W=V\begin{pmatrix}
        W_{1} &W_{2}\\
        0 &W_{3}
    \end{pmatrix}U^{*},
    \end{equation}
where $A_{1}, W_{1}\in \mathbb{C}^{r \times r}$ both non-singular, $A_{3}\in \mathbb{C}^{(m-r)\times(n-r)}$ and $W_{3}\in \mathbb{C}^{(n-r)\times(m-r)}$. The matrices $AW$ and $WA$ are given by
$$AW=U\begin{pmatrix}
        A_{1}W_{1} &A_{1}W_{2}+A_{2}W_{3}\\
        0 &A_{3}W_{3}
    \end{pmatrix}U^{*}=U\begin{pmatrix}
        C &E\\
        0 &N
    \end{pmatrix}U^{*}$$ and $$WA=V\begin{pmatrix}
        W_{1}A_{1} &W_{1}A_{2}+W_{2}A_{3}\\
        0 &W_{3}A_{3}
    \end{pmatrix}V^{*}=V\begin{pmatrix}
        R &S\\
        0 &T
    \end{pmatrix}V^{*},$$
where $C=A_{1}W_{1},~E=A_{1}W_{2}+A_{2}W_{3},~N=A_{3}W_{3},~R=W_{1}A_{1},~S=W_{1}A_{2}+W_{2}A_{3},\mbox{ and }T=W_{3}A_{3}$. This decomposition is known as weighted core-EP decomposition. In the below result, we discuss the representation of $W$-weighted $\{1,2,3,1^\kappa\}$-inverse based on the above decomposition.
\begin{theorem}\label{wcEPthm}
Let $A\in \mathbb{C}^{m \times n},~W\in \mathbb{C}^{n \times m}$ and $\kappa=\max\{\text{ind}(AW),\text{ind}(WA)\}$. 
Then $X$ is the $W$-weighted $\{1,2,3,1^\kappa\}$-inverse of $A$ if and only if
$$
X=U\begin{pmatrix}
        (W_{1}A_{1}W_{1})^{-1} &X_2\\
        0 &X_{4}
    \end{pmatrix}
    V^{*},$$ 
    where $(A_{1}W_{1})^{-1}((A_{1}W_{1})^\kappa E+\Hat{E}N)=\Hat{E}$, 
    $\Hat{E}=\sum_{i=0}^{k-1}(A_{1}W_{1})^{k-i}(A_{1}W_{2}+A_{2}W_{3})(A_{3}W_{3})^{i}$ and $\ X_2=X_2W_{3}NX_{4},\ X_{4}\in W_{3}N\{1,2,3\}.
    $
\end{theorem}
\begin{proof}
    Let
$ X=U \begin{pmatrix}
X_1& X_2\\
X_{3} & X_{4}
\end{pmatrix} V^{*}
$ be the $W$-weighted $\{1,2,3,1^k\}$-inverse of $A$ for suitable blocks.
Then $X$ will satisfy all the four conditions of Definition \ref{def3,1k}.
Notice that
\begin{align*}
   XW(AW)^{k+1}=U \begin{pmatrix}
X_1W_{1}C^{k+1}& X_1W_{1}(C^{k}E+\Hat{E}N)\\
X_{3}W_{1}C^{k+1} & X_{3}W_{1}(C^{k}E+\Hat{E}N)
\end{pmatrix} U^{*} \text{ and }(AW)^{k}=U \begin{pmatrix}
C^{k}& \Hat{E}\\
0 & 0
\end{pmatrix} U^{*}\end{align*}
 where
$$
\Hat{E}=\sum_{i=0}^{k-1}(A_{1}W_{1})^{k-i}(A_{1}W_{2}+A_{2}W_{3})(A_{3}W_{3})^{i}.
$$
The identity $XW(AW)^{\kappa+1}=(AW)^{\kappa}$ initiates \\
$X_1=(W_{1}A_{1}W_{1})^{-1},~(A_{1}W_{1})^{-1}(C^{\kappa}E+\Hat{E}N)=\Hat{E} \text{ and }X_{3}=0$.
Again
\begin{align*}
    WAWX=V \begin{pmatrix}
W_{1}CX_{1}& W_{1}CX_{2}+(W_{1}E+W_{2}N)X_{4}\\
0 & W_{3}NX_{4}
\end{pmatrix} V^{*}.
\end{align*}
Since the identity $(WAWX)^*=WAWX$ holds, then we obtain $W_{1}CX_{2}+(W_{1}E+W_{2}N)X_{4}=0$ and $X_{4} \in W_{3}N\{3\}$.
Also
\begin{align*}
  WAWXWAW=V \begin{pmatrix}
W_{1}C& W_{1}E+W_{2}N\\
0 & W_{3}NX_{4}W_{3}N
\end{pmatrix} U^{*}.
\end{align*}
The equality $WAWXWAW=WAW$ gives $X_{4} \in W_{3}N\{1\}$.

Further \begin{align*}
  XWAWX=U \begin{pmatrix}
X_1& X_2W_{3}NX_{4}\\
0 & X_{4}W_{3}NX_{4}
\end{pmatrix} V^{*}.
\end{align*}
Thus, from $XWAWX=X$, we get $X_2=X_2W_{3}NX_{4}$ and $X_{4} \in W_{3}N\{2\}$.
Therefore,
\begin{align*}
   X=U\begin{pmatrix}
        (W_{1}A_{1}W_{1})^{-1} &0\\
        0 &X_{4}
    \end{pmatrix}V^{*},
   \end{align*}
    where
    $$(A_{1}W_{1})^{-1}((A_{1}W_{1})^{k}E+\Hat{E}N)=\Hat{E},$$ 
    $$\Hat{E}=\sum_{i=0}^{k-1}(A_{1}W_{1})^{k-i}(A_{1}W_{2}+A_{2}W_{3})(A_{3}W_{3})^{i},\ ~X_2=X_2W_{3}NX_{4},\ X_{4} \in W_{3}N\{1,2,3\}.$$
\end{proof}

\begin{remark}
    If $X_{4}=0$ in Theorem \ref{wcEPthm}, then $X_2=0$ and hence $A^{(1,2,3,1^k)^W}$ coincides with $A^{\ep,W}$ {\rm \cite[Theorem 5.2]{fer18}}.
\end{remark}

Next Lemma continues the result (\cite{ben}, Chapter 2, Exercises 12,13).
Using this result we can talk of the uniqueness $\{1,2,3,1^k\}^W$-inverse for some integer $k$.
\begin{lemma}\label{lemma2.5}
Let $A,X \in \mathbb{C}^{m \times n}$. Then the followings hold:\\
$A\{W-1,2,3\}=\{(WAW)^{\dagger}+(I_{m}-(AW)^{\dagger}AW)M(WAW)^{\dagger}: M\in \mathbb{C}^{m \times m}\}$;\\
$A\{W-1,2,4\}=\{(WAW)^{\dagger}+(WAW)^{\dagger}N(I_{n}-WA(WA)^{\dagger}): N\in \mathbb{C}^{n \times n}\}$.
\end{lemma}
\begin{lemma}\label{lemma4.3}
Let $A\in \mathbb{C}^{m\times n}$.
If $X\in A\{1,2,3,1^k\}^W$ for some non-negative integer $k$, then
$$
\aligned
A^{\D,W}&=(XW)^{k+2}(AW)^{k}A,~A^{D,\dagger,W}=(WX)^{k+1}(WA)^{k+1}A^{\dagger},\\\
A^{c,\dagger,W}&=A^{\dagger}A(WX)^{k+1}(WA)^{k+1}A^{\dagger}.
\endaligned
$$
\end{lemma}

The subsequent theorem provides the existence of $W$-weighted $\{1,2,3,1^\kappa\}$-inverses of $A$.
\begin{theorem} \label{thm2.1}
Let $A \in \mathbb{C}^{m \times n},~ W\in \mathbb{C}^{n \times m}$ and $\kappa=\max\{\mathrm{ind}(AW),\mathrm{ind}(WA)\}$. Then
\begin{equation}\label{eqn5}
    \begin{alignedat}{2}
    A\{1,2,3,1^\kappa\}^W &=\left\{X+(I_{m}-XWAW)A^{\D,W}WAWX|~ X\in A\{1,2,3\}^W\right\}  \\
    &=\left\{X+(I_{m}-XWAW)YWAWX|~~ X \in A\{1,2,3\}^W\},~ Y \in A\{1^\kappa\}^W\right\},
  \end{alignedat}
\end{equation}
\begin{equation}
    \begin{alignedat}{2}
    A\{1,2,4,^\kappa1\}^W &=\left\{X+XWAWA^{D,W}(I_{n}-WAWX): X\in A\{1,2,3\}^W\right\}  \\
    &=\left\{X+XWAWY(I_{n}-WAWX)|~~ X\in A\{1,2,4\}^W,~~ Y\in A\{^\kappa\!1\}^W\right\}.
  \end{alignedat}
\end{equation}
Furthermore, set
\begin{center}
    $A^{\ominus,W}=A^{\core,W}+(I_{m}-A^{\core,W} WAW)A^{\D,W}WAWA^{\core,W}$, ~$A_{\ominus,W}=A^{\core,W}+A^{\core,W} WAWA^{\D,W}(I_{n}-WAWA^{\core,W} ).$
\end{center}
Then $A^{\ominus,W}\in A\{1,2,3,1^\kappa\}^W$,~$A_{\ominus,W}\in \{W-1,2,4,^\kappa1\}$.
\end{theorem}
}

\begin{theorem}\label{thm2.6}
Let $A\in \mathbb{C}^{m \times n}$ and $W \in \mathbb{C}^{n \times m}$ with $\kappa=\max\{\mathrm{ind}(AW),\mathrm{ind}(WA)\}$. Then:
\begin{equation} \label{eqn3}
    \begin{alignedat}{2}
    A\{1,2,3,1^\kappa\}^W&=\left\{A^{\ominus,W}+(I_{m}-(AW)^{\dagger}AW)M\right.\\ &{}\left.\left((WAW)^{\dagger}-(WAW)^{\dagger}WAWA^{\D,W}WAWA^{\core,W}\right): M\in \mathbb{C}^{m \times m}\right\}
  \end{alignedat}
\end{equation}
\begin{equation}
\begin{alignedat}{2}
    A\{1,2,4,^\kappa\!\! 1\}^W &=\{A_{\ominus,W}+\left((WAW)^{\dagger}-(WAW)^{\dagger}WAWA^{\D,W}WAWA^{\core,W}\right)\\ &\quad N(I_{n}-WA(WA)^{\dagger}): N\in \mathbb{C}^{n \times n}\}.
\end{alignedat}
\end{equation}
\end{theorem}

\begin{proof}
From Lemma \ref{lemma2.5}, $A$ has any $\{1,2,3\}^W$-inverse of the form $\{(WAW)^{\dagger}+(I_{m}-(AW)^{\dagger}AW)M(WAW)^{\dagger}\}$ for some $M\in \mathbb{C}^{m \times m}$.
Now substituting $X$, i.e., the  $\{1,2,3\}^W$-inverse of $A$ in (\ref{eqn5}) by the expression  $\{(WAW)^{\dagger}+(I_{m}-(AW)^{\dagger}AW)M(WAW)^{\dagger}\}$, we obtain
$$
\aligned
A\{1,2,3,1^\kappa\}^W &=\{(WAW)^{\dagger}+(I_{m}-(AW)^{\dagger}AW)M(WAW)^{\dagger}\\
&\ +(I_{m}-[(WAW)^{\dagger}+(I_m-(AW)^{\dagger}AW)M(WAW)^{\dagger}]WAW)A^{\D,W}WAWA^{\core,W}| M\in \mathbb{C}^{m \times m}\}.
\endaligned
$$
Now
\begin{multline*}
    (WAW)^{\dagger}+(I_{m}-(AW)^{\dagger}AW)M(WAW)^{\dagger}+\\\left(I_{m}-[(WAW)^{\dagger}+(I_{m}-(AW)^{\dagger}AW)M(WAW)^{\dagger}]WAW\right)A^{D,W}WAWA^{\core,W}\\ =(WAW)^{\dagger}+(I_{m}-(AW)^{\dagger}AW)M(WAW)^{\dagger}+A^{D,W}WAWA^{\core,W}-\\(WAW)^{\dagger}WAWA^{D,W}WAWA^{\core,W}-(I_{m}-(AW)^{\dagger}AW)M(WAW)^{\dagger}WAWA^{D,W}WAWA^{\core,W}\\ =(WAW)^{\dagger}+(I_{m}-(WAW)^{\dagger}WAW)A^{D,W}WAWA^{\core,W}+(I_{m}-(AW)^{\dagger}AW)M[(WAW)^{\dagger}\\-(WAW)^{\dagger}WAWA^{D,W}WAWA^{\core,W}\
    ]\\ =A^{\ominus,W}+(I_{m}-(AW)^{\dagger}AW)M\left((WAW)^{\dagger}-(WAW)^{\dagger}WAWA^{D,W}WAWA^{\core,W}\right),
\end{multline*}

   Thus
   \begin{equation*}
    \begin{alignedat}{2}
    A\{1,2,3,1^\kappa\}^W&=\{A^{\ominus,W}+(I_{m}-(AW)^{\dagger}AW)M\\ &\quad\left((WAW)^{\dagger}-(WAW)^{\dagger}WAWA^{D,W}WAWA^{\core,W}\right)|\ M\in \mathbb{C}^{m \times m}\}
  \end{alignedat}
\end{equation*}
Dually, we get
\begin{equation*}
\begin{alignedat}{2}
    A\{1,2,4,^\kappa1\}^W&=\{A_{\ominus,W}+\left((WAW)^{\dagger}-(WAW)^{\dagger}WAWA^{D,W}WAWA^{\core,W}\right)\\ &\quad N(I_{n}-WA(WA)^{\dagger})|\ N\in \mathbb{C}^{n \times n}\}.
\end{alignedat}
\end{equation*}
and complete the proof.
\end{proof}

The next example shows that if $\mathrm{ind}(AW)\geq 2$ then $\{1,2,3,1^\kappa\}^W$-inverse of $A$ is not unique.
\begin{example}
Choose $A=\begin{bmatrix}
-1 & 0& 0& 0& 0\\
0 & 1& 0& 0& 0\\
0 & 0& 0& 1& 0\\
0 & 0& 0& 0& -1
\end{bmatrix}$ and $W=\begin{bmatrix}
2 & 0 & 0& 0\\
0 & 1& 0& 0\\
0 & 0& 1& 1\\
0 & 0& 0& 1\\
0& 0& 0& 0
\end{bmatrix}$
with $\mathrm{ind}(AW)=2$.
Then
$A_{1}^{\ominus,W}=\begin{bmatrix}
    \frac{-1}{4}&0&0&0&0\\0 & 1& 0& 0&0\\0 & 0& 9& 0&0\\0 & 0& 1& 0&0
\end{bmatrix}$ and $A_{2}^{\ominus,W}=\begin{bmatrix}
    \frac{-1}{4}&0&0&0&0\\0 & 1& 0& 0&0\\0 & 0& 7& 0&0\\0 & 0& 1& 0&0
\end{bmatrix}$ are two $\{1,2,3,1^\kappa\}^W$-inverses of $A$.
\end{example}

The next example shows that if $\mathrm{ind}(WA)\geq 2$ then $\{1,2,4,^\kappa 1\}^W$-inverse of $A$ is not unique.
\begin{example}
    Let $A=\begin{bmatrix}
2 & 0 & 0& 0\\
0 & 1& 0& 0\\
0 & 0& 1& 1\\
0 & 0& 0& 1\\
0& 0& 0& 0
\end{bmatrix}$ and $W=\begin{bmatrix}
-1 & 0& 0& 0& 0\\
0 & 1& 0& 0& 0\\
0 & 0& 0& 1& 0\\
0 & 0& 0& 0& -1
\end{bmatrix}$ with $\mathrm{ind}(WA)=2$. Then $A^{'}_{\ominus,W}=\begin{bmatrix}
    \frac{1}{2}&0&\frac{1}{2}&\frac{-5}{4}\\0 & 1& 0&0\\0 & 0& 0& 0\\0 & 0& 0& 0\\0 & 0& -1& 3
\end{bmatrix}$ and $A^{''}_{\ominus,W}=\begin{bmatrix}
    \frac{1}{2}&0&\frac{1}{2}&\frac{-9}{4}\\0 & 1& 0&0\\0 & 0& 0& 0\\0 & 0& 0& 0\\0 & 0& -1& 5
\end{bmatrix}$ are two $\{1,2,4,^\kappa\! 1\}^W$-inverses of $A$.
\end{example}

For next two results, we recall the equivalent condition for a square matrix $A$ to be generalized EP \cite{kolihagEP} if $A^{D}A$ is Hermitian, and $A$ is called $k$-EP \cite{malik2016class} if $A^{k}A^\dagger=A^\dagger A^k$ where $k$=ind$(A)$.

\begin{proposition}
Let $A\in \mathbb{C}^{m \times n}$ and $W\in \mathbb{C}^{n \times m}$ with ind$(AW)=k$. Then $AW$ is generalized EP if and only if $((XW)^k(AW)^k)^*=(XW)^k(AW)^k$ for some $X\in A\{1,2,3,1^k\}^W$.
\end{proposition}
\begin{proof}
From Lemma \ref{lemma4.3}, for any $X\in A\{1,2,3,1^k\}^W$ we obtain $A^{D,W}=(XW)^{k+2}(AW)^{k}A$.
Thus,
$$
(AW)^\D(AW)=A^{\D,W}WAW=(XW)^{k+2}(AW)^{k+2}=(XW)^{k}(AW)^k.
$$
Hence, $AW$ is generalized EP if and only if $((XW)^k(AW)^k)^*=(XW)^k(AW)^k$.
\end{proof}

\begin{proposition}
Let $A\in \mathbb{C}^{m \times n},~W\in \mathbb{C}^{n \times m}$ and $\kappa=\max\{\mathrm{ind}(AW),\mathrm{ind}(WA)\}$.
If $AW$ and $WA$ are both k-EP, then the following hold:
\begin{enumerate}
    \item[\rm (i)] $(WAW)^{\dagger}=A^{\ominus,W}=A_{\ominus,W}$;
    \item[\rm (ii)] $A^{\ominus,W}=A_{\ominus,W}$;
    \item[\rm (iii)]  $A\{1,2,3,1^\kappa\}^{W} \cap A\{1,2,4,^\kappa\! 1\}^{W}\neq \emptyset$.
\end{enumerate}
\end{proposition}
\begin{proof}
$(i)$ Since $AW$ is $m$-EP, it follows $(AW)^{\dagger}(AW)^{\kappa+1}=(AW)^\kappa=(AW)^{\kappa+1}(AW)^{\dagger}$ from \cite{malik2016class}.
Hence,
$$
\aligned
A^{\ominus,W}&=(WAW)^{\dagger}+(I_{m}-(WAW)^{\dagger}WAW)A^{\D,W}WAW(WAW)^{\dagger}\\
&=(WAW)^{\dagger}+(I_{m}-(WAW)^{\dagger}WAW)[(AW)^k((AW)^\D)^\kappa](WAW)^{\dagger}\\
&=(WAW)^{\dagger}+[(AW)^\kappa-(AW)^{\dagger}(AW)^{\kappa+1}]((AW)^\D)^\kappa(WAW)^{\dagger}=(WAW)^{\dagger}
\endaligned
$$
and similarly $(WAW)^{\dagger}=A_{\ominus,W}$ is dual.

\smallskip
(ii) and (iii) are verified analogously.
\end{proof}

\section{W-index-MP, W-MP-index and W-MP-index-MP matrix }
In this section, we extend the notion of index-MP, MP-index, and MP-index-MP matrices \cite{indexmp} for rectangular complex matrices. 
\begin{theorem}\label{mainthm}
Let $A\in \mathbb{C}^{m \times n},~W\in \mathbb{C}^{n \times m}$, and $\kappa=\max\{\mathrm{ind}(AW),\mathrm{ind}(WA)\}$. Then  
\begin{enumerate}
    \item[(i)] $(WA)^{\kappa}WAA^{\dagger}$ is the unique solution to the following system 
    \begin{equation}\label{18theqn}        XA((WA)^D)^{\kappa+1}X=X,~A((WA)^D)^{\kappa+1}X=A(WA)^{D}WAA^{\dagger}~\text{and }XA=(WA)^{\kappa+1}.
    \end{equation}
    \item[(ii)] $A^{\dagger}AW(AW)^{\kappa}$ is the unique solution to the following system 
    \begin{equation*}\label{19theqn}
X((AW)^D)^{\kappa+1}AX=X,~X((AW)^D)^{\kappa+1}A=A^{\dagger}AW(AW)^{D}A~\text{and }AX=(AW)^{\kappa+1}.
    \end{equation*}
    \item[(iii)] $A^{\dagger}AW(AW)^{\kappa}AA^{\dagger}$ is the unique solution to the following system 
    \begin{equation*}\label{20theqn}
XA((WA)^D)^{\kappa+1}X=X,~XA((WA)^D)^{\kappa+1}=A^{\dagger}AWA(WA)^{D}~\text{and }AX=(AW)^{\kappa+1}AA^{\dagger}.
    \end{equation*}
\end{enumerate}
\end{theorem}
\begin{proof}
(i) Let $X=(WA)^{\kappa}WAA^{\dagger}$. Then 
\[A((WA)^D)^{\kappa+1}X=A((WA)^D)^{\kappa+1}(WA)^{\kappa}WAA^{\dagger}=A(WA)^{D}WAA^{\dagger,}\] 
\begin{eqnarray*}
    XA((WA)^D)^{\kappa+1}X&=&XA(WA)^{D}WAA^{\dagger}=(WA)^{\kappa}W(AA^{\dagger}A)(WA)^{D}WAA^{\dagger}\\
    &=&((WA)^{\kappa}WA(WA)^{D})WAA^{\dagger}=(WA)^{\kappa}WAA^{\dagger}=X, 
\end{eqnarray*} 
and $XA=(WA)^{\kappa}WAA^{\dagger}A=(WA)^{\kappa+1}$. Thus  $X=(WA)^{\kappa}WAA^{\dagger}$ satisfies the system \eqref{18theqn}. To show the uniqueness, suppose $X$ and $Y$ be two solutions to the system \eqref{18theqn}, then 
\[X=X(A((WA)^D)^{\kappa+1}X)=(XA)((WA)^D)^{\kappa+1}Y=(WA)^{\kappa+1}((WA)^D)^{\kappa+1}Y=YA((WA)^D)^{\kappa+1}Y=Y.\]
(ii) and (iii) can be proved similarly.
\end{proof}
In view of Theorem \ref{mainthm}, next we define the weighted version of index-MP, MP-index, and MP-index-MP inverse.
\begin{definition}
    Let $A\in \mathbb{C}^{m \times n},~W\in \mathbb{C}^{n \times m}$, and $\kappa=\max\{\mathrm{ind}(AW),\mathrm{ind}(WA)\}$. 
\begin{enumerate}
    \item[(i)] The W-$\kappa$ index-MP matrix (or in short W-$\kappa$-MP) of $A$ is defined as $$A^{\kappa,\dagger,W}=(WA)^{\kappa}WAA^{\dagger}=(WA)^{\kappa+1}A^{\dagger}.$$
    \item[(ii)] The W-MP-$\kappa$ index inverse (or in short W-MP-$\kappa$) of $A$ is defined as $$A^{\dagger,\kappa,W}=A^{\dagger}AW(AW)^{\kappa}=A^{\dagger}(AW)^{\kappa+1}.$$
    \item[(iii)] The W-MP-$\kappa$-MP index inverse (or in short W-MP-$\kappa$-MP) of $A$ is defined  as $$A^{\dagger,\kappa,\dagger,W}=A^{\dagger}AW(AW)^{\kappa}AA^{\dagger}=A^{\dagger}A(WA)^{\kappa+1}A^{\dagger}=A^{\dagger}(AW)^{\kappa+1}AA^{\dagger}.$$
\end{enumerate}
\end{definition}
In the following example, we could see that the W-MP-$\kappa$-index, W-$\kappa$-index-MP, and MP-$\kappa$-index-MP matrices of $A$ (for any $\kappa =\mathrm{max\{ind(AW),ind(WA)\}}\geq 1$) are in general different from its weighted Drazin inverse, Moore–Penrose, weighted DMP, weighted MPD inverse, and weighted CMP inverse. 
\begin{example}\rm
    Let $A=\begin{bmatrix}
1 &2 & 0 & 0& -1\\
-1 & 0 & 1& 0& 1\\
0 & 2 & 0& -1& 0
\end{bmatrix}$ and $W=\begin{bmatrix}
1 & 2 & 0\\
0 & 0 & 1\\
1 & 2 & 1\\
0 & 0 & 0\\
0 & 0 & 0
\end{bmatrix}$. We can find  $\mathrm{ind}(AW)=1$ and $\mathrm{ind}(WA)=1$. So $\kappa=\mathrm{max}\{1,1\}=1$. Now we evaluate  \[A^{\dagger}=\begin{bmatrix}
    0.2273&-0.1818&-0.1818\\
    0.2727&0.1818&0.1818\\
    0.4545&0.6364&-0.3636\\
    0.5455&0.3636&-0.6364\\
   -0.2273&0.1818&0.1818
\end{bmatrix},~A^{D,W}=\begin{bmatrix}
    -1 &-1 & 2 & \frac{5}{2}& 1\\
    0 & \frac{1}{2} & 0& -\frac{1}{4}& 0\\
    0 & 1 & 0& -\frac{1}{2}& 0
\end{bmatrix},\] 
\[A^{D,\dagger,W}=\begin{bmatrix}
    1 & 2 & 0\\
    0 & 0 & 1\\
    1 & 2 & 1\\
    0 & 0 & 0\\
    0 & 0 & 0
\end{bmatrix},~ A^{\dagger,D,W}=A^{c,\dagger,W}=\begin{bmatrix}
    0.2273&0.4545&-0.0909\\
    0.2727&0.5455&1.0909\\
    0.4545&0.9091&0.8182\\
    0.5455&1.0909&0.1818\\
   -0.2273&-0.4545&0.0909
\end{bmatrix}.\]
But $A^{2,\dagger,W}=\begin{bmatrix}
    1 & 2 & 4\\
    0 & 0 & 2\\
    1 & 2 & 6\\
    0 & 0 & 0\\
    0 & 0 & 0
\end{bmatrix}$,  $A^{\dagger,2,W}=A^{\dagger,2,\dagger,W}=\begin{bmatrix}
    0.2273&0.4545&0.7273\\
    0.2727&0.5455&3.2727\\
    0.4545&0.9091&3.4545\\
    0.5455&1.0909&2.5455\\
   -0.2273&-0.4545&-0.7273
\end{bmatrix}$.
\end{example}
So, it observed that these inverses produce a new class of weighted generalized inverses. 
\begin{lemma}
   Let $A\in \mathbb{C}^{m \times n},~W\in \mathbb{C}^{n \times m}$, and $\kappa=\max\{\mathrm{ind}(AW),\mathrm{ind}(WA)\}$.  Then
\begin{enumerate}
    \item[(i)] $A^{\kappa,\dagger,W}=(WA)^{k+1} A^{D,\dagger,W}$.
    \item[(ii)] $A((WA)^D)^{\kappa+1}A^{\kappa,\dagger,W}=AA^{D,\dagger,W}$.
    \item[(iii)] $A^{\dagger,\kappa,\dagger,W}A((WA)^D)^{\kappa+1}=A^{\dagger}AWA(WA)^{D}=A^{\dagger,\kappa,W}A((WA)^D)^{\kappa+1}$.
    \item[(iv)] $A^{\dagger,\kappa,W}((AW)^D)^{\kappa+1}A= A^{\dagger,D,W}A$.
    \item[(v)] $A^{\dagger,\kappa,W}= A^{\dagger,D,W}(AW)^{k+1}$.
    \item[(vi)] $A^{\dagger,\kappa,\dagger,W}=A^{\dagger,\kappa,W}AA^{D,\dagger,W}=A^{\dagger,D,W}AA^{\kappa,\dagger,W}=A^{\dagger,D,W}A(WA)^{\kappa+1}A^{D,\dagger,W}$.
    \item[(vii)] $((AW)^D)^{\kappa+1}AA^{\dagger,\kappa,\dagger,W}= (AW)^{D}AWAA^{\dagger}$.
\end{enumerate}
\end{lemma}
\begin{proof}
Using the definitions of $A^{\kappa,\dagger,W}$, $A^{\dagger,\kappa,W}$, $A^{\dagger,\kappa,\dagger,W}$ and $WA^{D,W}=(WA)^{D}$, we can verify all the parts.
\end{proof} 
An equivalent characterization of the $W$-$\kappa$-MP matrix is presented in the next result.
\begin{theorem}\label{indexMP}
  Let $A\in \mathbb{C}^{m \times n},~W\in \mathbb{C}^{n \times m}$, and $\kappa=\max\{\mathrm{ind}(AW),\mathrm{ind}(WA)\}$.  Then the following statements are equivalent:
\begin{enumerate}
    \item[(i)] $X=A^{\kappa,\dagger,W}$.
    \item[(ii)] $A((WA)^D)^{\kappa+1}XA((WA)^D)^{\kappa+1}=A((WA)^D)^{\kappa+1},~XA((WA)^D)^{\kappa+1}X=X,~A((WA)^D)^{\kappa+1}X=\\A(WA)^{D}WAA^{\dagger}\text{ and }XA=(WA)^{\kappa+1}$.
    \item[(iii)] $XA((WA)^D)^{\kappa+1}X=X,~A((WA)^D)^{\kappa+1}X=A(WA)^{D}WAA^{\dagger}\text{ and }XA((WA)^D)^{\kappa+1}=WA^{D,W}WA$.
    \item[(iv)] $WA^{D,W}WAX=X\text{ and }AX=(AW)^{\kappa+1}AA^{\dagger}$.
    \end{enumerate}
\end{theorem}
\begin{proof}
(i)$\Rightarrow$(ii) It is sufficient to show the first condition only. Let $X=A^{\kappa,\dagger,W}$. Then
\begin{eqnarray*}
A((WA)^D)^{\kappa+1}XA((WA)^D)^{\kappa+1}&=&A((WA)^D)^{\kappa+1}A^{\kappa,\dagger,W}A((WA)^D)^{\kappa+1}\\
&=&A((WA)^D)^{\kappa+1}(WA)^{\kappa+1}A^{\dagger}A((WA)^D)^{\kappa+1}\\
&=&A(WA)^{D}W(AA^{\dagger}A)((WA)^D)^{\kappa+1}=A((WA)^D)^{\kappa+1}.   
\end{eqnarray*}
(ii)$\Rightarrow$(iii) Post multiplying $((WA)^D)^{\kappa+1}$ to the equation $XA=(WA)^{\kappa+1}$, we obtain
\[XA((WA)^D)^{\kappa+1}=(WA)^{\kappa+1}((WA)^D)^{\kappa+1}=(WA)^{D}WA=WA^{D,W}WA.\]
(iii)$\Rightarrow$(iv) Assume $(iii)$ holds. Then $X=(XA((WA)^D)^{\kappa+1})X=WA^{D,W}WAX$ and 
\begin{eqnarray*}
AX&=&AWA^{D,W}WAX=A(WA)^{\kappa+1}((WA)^D)^{\kappa+1}X=(AW)^{\kappa+1}A((WA)^D)^{\kappa+1}X\\
&=&(AW)^{\kappa+1}A(WA)^{D}WAA^{\dagger}=A(WA)^{\kappa+1}(WA)^{D}WAA^{\dagger}=(AW)^{\kappa+1}AA^{\dagger}.  
\end{eqnarray*}
(iv)$\Rightarrow$(i) It follows from the below expression: 
    \begin{eqnarray*}
        X&=&WA^{D,W}W(AX)=(WA)^{D}W(AW)^{\kappa+1}AA^{\dagger}=(WA)^{D}(WA)^{\kappa+1}WAA^{\dagger}\\
        &=&(WA)^{\kappa}WAA^{\dagger}=A^{\kappa,\dagger,W}.
    \end{eqnarray*}
\end{proof}
\begin{corollary}
     Let $A\in \mathbb{C}^{m \times n},~W\in \mathbb{C}^{n \times m}$, and $\kappa=\max\{\mathrm{ind}(AW),\mathrm{ind}(WA)\}$.  Then the following statements are equivalent:
\begin{enumerate}
    \item[(i)] $X=A^{\kappa,\dagger,W}$.
    \item[(ii)] $WA^{D,W}WAX=X~\text{ and }A((WA)^D)^{\kappa+1}X=AWA^{D,W}WAA^{\dagger}$.
    \item[(iii)] $XAA^{\dagger}=X\text{ and }XA=(WA)^{\kappa+1}$.
    \item[(iv)] $XA(WA)^{D}WAA^{\dagger}=X~\text{ and }XA((WA)^D)^{\kappa+1}=WA^{D,W}WA$.
    \item[(v)] $XA((WA)^D)^{\kappa+1}(WA)^{\kappa}=(WA)^{\kappa}\text{ and }XA(WA)^{D}WAA^{\dagger}=X$.
    \item[(vi)] $WA^{D,W}WAX=X\text{ and }(WA)^{D}X=(WA)^{\kappa}A^{\dagger}$.
\end{enumerate}
\end{corollary}
Similarly, we can prove the following results for W-MP-$\kappa$, and W-MP-$\kappa$-MP matrices.
\begin{theorem}
     Let $A\in \mathbb{C}^{m \times n},~W\in \mathbb{C}^{n \times m}$, and $\kappa=\max\{\mathrm{ind}(AW),\mathrm{ind}(WA)\}$.  Then the following statements are equivalent:
\begin{enumerate}
    \item[(i)] $X=A^{\dagger,\kappa,W}$.
    \item[(ii)] $((AW)^D)^{\kappa+1}AX((AW)^D)^{\kappa+1}A=((AW)^D)^{\kappa+1}A,~X((AW)^D)^{\kappa+1}AX=X,~X((AW)^D)^{\kappa+1}A=\\A^{\dagger}AW(AW)^{D}A~\text{ and }AX=(AW)^{\kappa+1}$.
    \item[(iii)] $X((AW)^D)^{\kappa+1}AX=X,~X((AW)^D)^{\kappa+1}A=A^{\dagger}AW(AW)^{D}A\text{ and }((AW)^D)^{\kappa+1}AX=AWA^{D,W}W$.
    \item[(iv)] $XAWA^{D,W}W=X\text{ and }XA=A^{\dagger}A(WA)^{\kappa+1}$.
    \end{enumerate}
\end{theorem}
\begin{corollary}
  Let $A\in \mathbb{C}^{m \times n},~W\in \mathbb{C}^{n \times m}$, and $\kappa=\max\{\mathrm{ind}(AW),\mathrm{ind}(WA)\}$.  Then the following statements are equivalent:
  \begin{enumerate}
       \item[(i)] $X=A^{\dagger,\kappa,W}$.
     \item[(ii)] $XAWA^{D,W}W=X\text{ and }X((AW)^D)^{\kappa+1}A=A^{\dagger}AW(AW)^{D}A$.
    \item[(iii)] $A^{\dagger}AX=X\text{ and }AX=(AW)^{\kappa+1}$.
    \item[(iv)] $A^{\dagger}AW(AW)^{D}AX=X\text{ and }((AW)^D)^{\kappa+1}AX=AWA^{D,W}W$.
    \item[(v)] $(AW)^{\kappa}((AW)^D)^{\kappa+1}AX=(AW)^{\kappa}\text{ and }A^{\dagger}AW(AW)^{D}AX=X$.
    \item[(vi)] $XAWA^{D,W}W=X\text{ and }X(AW)^{D}=A^{\dagger}(AW)^{\kappa}$.
  \end{enumerate}     
\end{corollary}
\begin{theorem}
  Let $A\in \mathbb{C}^{m \times n},~W\in \mathbb{C}^{n \times m}$, and $\kappa=\max\{\mathrm{ind}(AW),\mathrm{ind}(WA)\}$.  Then the following statements are equivalent:
\begin{enumerate}
    \item[(i)] $X=A^{\dagger,\kappa,\dagger,W}$.
    \item[(ii)] $X((AW)^D)^{\kappa+1}AX=X,~((AW)^D)^{\kappa+1}AX=(AW)^{D}AWAA^{\dagger}\text{ and }AX=(AW)^{\kappa+1}AA^{\dagger}$.
    \item[(iii)] $X((AW)^D)^{\kappa+1}AX=X,~((AW)^D)^{\kappa+1}AX=(AW)^{D}AWAA^{\dagger}\text{ and }XA((WA)^D)^{\kappa+1}=A^{\dagger}AWA(WA)^{D}$.
    \item[(iv)] $A^{\dagger}AX=X\text{ and }AX=(AW)^{\kappa+1}AA^{\dagger}$.
  \end{enumerate}
\end{theorem}
\begin{corollary}
Let $A\in \mathbb{C}^{m \times n},~W\in \mathbb{C}^{n \times m}$, and $\kappa=\max\{\mathrm{ind}(AW),\mathrm{ind}(WA)\}$.  Then the following statements are equivalent:
    \begin{enumerate}
        \item[(i)] $X=A^{\dagger,\kappa,\dagger,W}$.
        \item[(ii)] $XA=A^{\dagger}A(WA)^{\kappa+1}\text{ and }XAA^{\dagger}=X$.
    \item[(iii)] $X(AW)^{D}AWAA^{\dagger}=X\text{ and }XA((WA)^D)^{\kappa+1}=A^{\dagger}AWA(WA)^{D}$.
    \item[(iv)] $X(AW)^{D}AWAA^{\dagger}=X\text{ and }XA((WA)^D)^{\kappa+1}(WA)^{\kappa}=A^{\dagger}A(WA)^{\kappa}$.
    \item[(v)] $A^{\dagger}(AW)^{D}AWAX=X\text{ and }((AW)^{D})^{\kappa+1}AX=(AW)^{D}AWAA^{\dagger}$.
    \item[(vi)] $A^{\dagger}(AW)^{D}AWAX=X\text{ and }(AW)^{\kappa}((AW)^{D})^{\kappa+1}AX=(AW)^{\kappa}AA^{\dagger}$.
    \end{enumerate}
\end{corollary}
The representation of W-$\kappa$-MP, W-MP-$k\kappa$ and W-MP-$\kappa$-MP matrices based on core-EP decomposition \eqref{Wcep-decompeqn} is presented in the following theorem.
\begin{theorem}
    Let $A\in \mathbb{C}^{m \times n},~W\in \mathbb{C}^{n \times m}$ and $\kappa=\max\{\mathrm{ind}(AW),\mathrm{ind}(WA)\}$.  Then 
\[A^{\kappa,\dagger,W}=V\begin{pmatrix}
        R^{\kappa+1}A_{1}^{*}\Delta +R^{\kappa}S((I-A_{3}^{\dagger}A_{3})A_{2}^{*}\Delta) &-R^{\kappa+1}A_{1}^{*}\Delta A_{2}A_{3}^{\dagger}+R^{\kappa}S(A_{3}^{\dagger}-(I-A_{3}^{\dagger}A_{3})A_{2}^{*}\Delta A_{2}A_{3}^{\dagger})\\
          0&0
    \end{pmatrix}U^{*},\]
\[A^{\dagger,\kappa,W}=V\begin{pmatrix}
        A_{1}^{*}\Delta C^{\kappa+1} & A_{1}^{*}\Delta \Hat{E}\\
        (I-A_{3}^{\dagger}A_{3})A_{2}^{*}\Delta C^{\kappa+1} &(I-A_{3}^{\dagger}A_{3})A_{2}^{*}\Delta \Hat{E}
    \end{pmatrix}U^{*},\]
and 
\[A^{\dagger,\kappa,\dagger,W}=V\begin{pmatrix}
        A_{1}^{*}\Delta C^{\kappa+1} & 0 \\
        (I-A_{3}^{\dagger}A_{3})A_{2}^{*}\Delta C^{\kappa+1} & 0
    \end{pmatrix}U^{*},\]
 where $C=A_{1}W_{1},~R=W_{1}A_{1},~S=W_{1}A_{2}+W_{2}A_{3},~\Hat{E}=\sum_{i=0}^{k-1}(A_{1}W_{1})^{k-i}(A_{1}W_{2}+A_{2}W_{3})(A_{3}W_{3})^{i}$ and $\Delta=(A_{1}A_{1}^{*}+A_{2}(I-A_{3}^{\dagger}A_{3})A_{2}^{*})^{-1}$.
\end{theorem}
\begin{proof}
    In \cite{Mptriangular}, it has been proved that $A^{\dagger}$ will be given by 
\[A^{\dagger}=V\begin{pmatrix}
        A_{1}^{*}\Delta &-A_{1}^{*}\Delta A_{2}A_{3}^{\dagger}\\
        (I-A_{3}^{\dagger}A_{3})A_{2}^{*}\Delta  &A_{3}^{\dagger}-(I-A_{3}^{\dagger}A_{3})A_{2}^{*}\Delta A_{2}A_{3}^{\dagger}
    \end{pmatrix}U^{*},\]
    where $\Delta=(A_{1}A_{1}^{*}+A_{2}(I-A_{3}^{\dagger}A_{3})A_{2}^{*})^{-1}$. Using this representation, we can verify the theorem. 
\end{proof}

The relation of outer inverse with specified range and null space along with W-$\kappa$-MP, W-MP-$\kappa$ and W-MP-$\kappa$-MP matrices is discussed in the next results.
\begin{theorem}\label{thmprojection}
 Let $A\in \mathbb{C}^{m \times n},~W\in \mathbb{C}^{n \times m}$, and $\kappa=\max\{\mathrm{ind}(AW),\mathrm{ind}(WA)\}$. Then 
\begin{enumerate}
    \item[(i)] $A((WA)^D)^{\kappa+1}A^{\kappa,\dagger,W}=P_{\rg((WA)^{\kappa}),\nl(A^{D,W}A^{\dagger})}$.
    \item[(ii)] $A^{\kappa,\dagger,W}A((WA)^D)^{\kappa+1}=P_{\rg((WA)^{\kappa}),\nl((WA)^{\kappa})}$.
    \item[(iii)] $A^{\kappa,\dagger,W}=[A((WA)^D)^{\kappa+1}]^{(2)}_{\rg((WA)^{\kappa}),\nl(A^{D,W}A^{\dagger})}=[A((WA)^D)^{\kappa+1}]^{(1,2)}_{\rg((WA)^{\kappa}),\nl(A^{D,W}A^{\dagger})}$.
    \end{enumerate}
\end{theorem}
\begin{proof}
(i) Clearly $A((WA)^D)^{\kappa+1}A^{\kappa,\dagger,W}$ is a projector since  $A^{\kappa,\dagger,W}A((WA)^D)^{\kappa+1}A^{\kappa,\dagger,W}=A^{\kappa,\dagger,W}$. From $A^{D,W}=A^{D,W}WAWA^{D,W}=A^{D,W}WAA^{\dagger}AWA^{D,W}$ and $A((WA)^D)^{\kappa+1}A^{\kappa,\dagger,W}=AWA^{D,W}WAA^{\dagger}$, we obtain \[\rg(A((WA)^D)^{\kappa+1}A^{\kappa,\dagger,W})\subseteq \rg(AWA^{D,W})=\rg(AWA^{D,W}WAA^{\dagger}AWA^{D,W})\subseteq\rg(A((WA)^D)^{\kappa+1}A^{\kappa,\dagger,W})\]
and subsequently, 
\[\rg(A((WA)^D)^{\kappa+1}A^{\kappa,\dagger,W})= \rg(AWA^{D,W})=\rg(A^{D,W})=\rg(((WA)^{\kappa}).\]
By Theorem 3.1 \cite{DMPRectangular}, we have $\nl(A((WA)^D)^{\kappa+1}A^{\kappa,\dagger,W})=\nl(AWA^{D,W}WAA^{\dagger})=\nl(A^{D,W}A^{\dagger})$.\\
(ii) Using $A^{\kappa,\dagger,W}A((WA)^D)^{\kappa+1}=WA(WA)^{D}$, we can show that 
\[\rg(A^{\kappa,\dagger,W}A((WA)^D)^{\kappa+1})=\rg(WA(WA)^{D})=\rg((WA)^{\kappa})\]
   and 
   \[\nl(A^{\kappa,\dagger,W}A((WA)^D)^{\kappa+1})=\nl(WA(WA)^{D})=\nl((WA)^{\kappa}).\]
 (iii)  From $\rg(A^{\kappa,\dagger,W})=\rg(A^{\kappa,\dagger,W}A(WA)^D)^{\kappa+1})=\rg(WA(WA)^{D})=\rg((WA)^{\kappa})$ and by part (i), we have $\nl(A^{\kappa,\dagger,W})=\nl(A((WA)^D)^{\kappa+1}A^{\kappa,\dagger,W})=\nl(A^{D,W}A^{\dagger})$.
\end{proof}
\begin{corollary}\label{corprojection}
 Let $A\in \mathbb{C}^{m \times n},~W\in \mathbb{C}^{n \times m}$, and $\kappa=\max\{\mathrm{ind}(AW),\mathrm{ind}(WA)\}$. Then 
\begin{enumerate}
       \item[(i)] $((AW)^D)^{\kappa+1}AA^{\dagger,\kappa,W}=P_{\rg((AW)^{\kappa}),\nl((AW)^{\kappa})}$.
    \item[(ii)] $A^{\dagger,\kappa,W}((AW)^D)^{\kappa+1}A=P_{\rg(A^{\dagger}A^{D,W}),\nl((WA)^{\kappa})}$.
    \item[(iii)] $A^{\dagger,\kappa,W}=[A((WA)^D)^{\kappa+1}]^{(2)}_{\rg(A^{\dagger}A^{D,W}),\nl((AW)^{\kappa})}=[A((WA)^D)^{\kappa+1}]^{(1,2)}_{\rg(A^{\dagger}A^{D,W}),\nl((AW)^{\kappa})}$.
    \end{enumerate}
\end{corollary}
\begin{corollary}\label{cor1projection}
 Let $A\in \mathbb{C}^{m \times n},~W\in \mathbb{C}^{n \times m}$, and $\kappa=\max\{\mathrm{ind}(AW),\mathrm{ind}(WA)\}$. Then 
\begin{enumerate}
       \item[(i)] $((AW)^D)^{\kappa+1}AA^{\dagger,\kappa,\dagger,W}=P_{\rg((AW)^{\kappa}),\nl(A^{D,W}A^{\dagger})}$.
    \item[(ii)] $A^{\dagger,\kappa,\dagger,W}A((WA)^D)^{\kappa+1}=P_{\rg(A^{\dagger}A^{D,W}),\nl((WA)^{\kappa})}$.
    \item[(iii)] $\text{rank}(A^{\kappa,\dagger,W})=\text{rank}(A^{\dagger,\kappa,W})=\text{rank}(A((WA)^D)^{\kappa+1})$.
\end{enumerate}
\end{corollary}
\begin{theorem}
  Let $A\in \mathbb{C}^{m \times n},~W\in \mathbb{C}^{n \times m}$, and $\kappa=\max\{\mathrm{ind}(AW),\mathrm{ind}(WA)\}$. Then 
\begin{enumerate}
    \item[(i)] $X=A^{\kappa,\dagger,W}$ is the unique solution to the system 
    \begin{equation}\label{eqn19}
A((WA)^D)^{\kappa+1}X=P_{\rg((WA)^{\kappa}),\nl(A^{D,W}A^{\dagger})} \text{ and } \rg(X)\subseteq \rg((WA)^{\kappa};
    \end{equation}
    \item[(ii)] $X=A^{\dagger,\kappa,W}$ is the unique solution to the system 
    \begin{equation*}
((AW)^D)^{\kappa+1}AX=P_{\rg((AW)^{D}),\nl((WA)^{\kappa})} \text{ and } \rg(X)\subseteq \rg(A^{\dagger}A^{D,W}).
    \end{equation*} 
\end{enumerate}
\end{theorem}
\begin{proof},
(i) By Theorem \ref{thmprojection} (i), it is clear that,  $A^{\kappa,\dagger,W}$ satisfies equation \eqref{eqn19}. Next, we will show the uniqueness. Suppose exists two matrices $X$ and $Y$, which satisfy \eqref{eqn19}. From
\[A((WA)^D)^{\kappa+1}(X-Y)=P_{\rg((WA)^{\kappa}),\nl(A^{D,W}A^{\dagger})}-P_{\rg((WA)^{\kappa}),\nl(A^{D,W}A^{\dagger})}=0,\]
we obtain $\rg(X-Y)\subseteq \nl(A((WA)^D)^{\kappa+1})=\nl((WA)^{\kappa})$.  Applying $\rg(X)\subseteq \rg((WA)^{\kappa}$ and $\rg(Y)\subseteq \rg((WA)^{\kappa}$, we have $\rg(X-Y)\subseteq \rg((WA)^{\kappa} \cap \nl((WA)^{\kappa})=\{0\}$. Thus $X=Y$ and completes the proof.\\
(ii) The proof is similar to part (i).
\end{proof}

\section{Weighted Generalized bilateral inverse}

The authors in \cite{malik} introduced the dual of DMP  inverse by interchanging the roles of the Drazin and Moore-Penrose inverse of $A \in  \mathbb{C}^{n\times n}$ within the definition of DMP inverse.
In \cite{bilateral}, the authors investigated expressions involving two inner and/or outer inverses of $A$.
In this section, we discuss the concept of weighted generalized bilateral inverse of a rectangular matrix $A\in \mathbb{C}^{m \times n}$. 
Furthermore, we establish the relationships between weighted generalized bilateral inverses and some well-known generalized inverses.

\begin{definition}
    Let $A\in \mathbb{C}^{m \times n}$, $W\in \mathbb{C}^{n \times m}$ and $X_1,~X_2\in A\{1^W\}\cup A\{2^W\}$. Then $X_1WAWX_{2}:=(WAW)^{ X_1\looparrowright X_2}$ is called weighted generalized bilateral inverse of $A$.
\end{definition}

We start with these three lemmas that will be useful in this paper.
\begin{lemma}
    Let $A\in \mathbb{C}^{m \times n}$ and let $X_1,X_2\in A\{1^W,2^W\}$. 
    Then $(WAW)^{ X_1\looparrowright X_2}$ and $(WAW)^{ X_2\looparrowright X_1}$ both belong to $A\{1^W,2^W\}$.
\end{lemma}
\begin{proof}
    For $X_1,X_2\in    A\{1^W,2^W\}$, let us notice 
\begin{align*}
        AW (WAW)^{ X_1\looparrowright X_2} WA &=(AWX_{1}WA)WX_{2}WA =AWX_{2}WA=A
            \end{align*}
and
\begin{align*}            
        (WAW)^{ X_1\looparrowright X_2} WAW (WAW)^{ X_1\looparrowright X_2} &=X_1 W(AWX_{2}WA)W (WAW)^{ X_1\looparrowright X_2}\\
        &= (X_1WAWX_{1})WAWX_{2}\\
        &=(WAW)^{ X_1\looparrowright X_2}.
    \end{align*}
    Therefore, $(WAW)^{ X_1\looparrowright X_2} \in A\{1^W,2^W\}$. 
    
Following similar principles we obtain $(WAW)^{ X_2\looparrowright X_1} \in A\{1^W,2^W\}$.
\end{proof}
\begin{lemma}\label{outer}
    Let $A\in \mathbb{C}^{m \times n}$ and let $X_1\in A\{2^W\}$ and $X_2\in A\{1^W\}$. 
Then $(WAW)^{X_1\looparrowright X_2}$ and $(WAW)^{X_2\looparrowright X_1}$ both belong to $A\{2^W\}$.
\end{lemma}
\begin{proof}
    Since $X_1\in A\{2^W\}$ and $X_2\in A\{1^W\}$, it follows
    \begin{align*}
        (WAW)^{X_1\looparrowright X_2} WAW (WAW)^{X_1\looparrowright X_2} &=X_1W(AWX_{2}WA)W (WAW)^{X_1\looparrowright X_2}\\
        &= (X_1WAWX_{1})WAWX_{2}\\
        &=(WAW)^{X_1\looparrowright X_2}
        \end{align*}
and
    \begin{align*}
        (WAW)^{X_2\looparrowright X_1}WAW (WAW)^{X_2\looparrowright X_1}  &=(WAW)^{X_2\looparrowright X_1}W(AWX_{2}WA)WX_{1}\\
        &= X_2WAW(X_1WAWX_{1})\\
        &=(WAW)^{X_2\looparrowright X_1}.
    \end{align*}
Therefore, $(WAW)^{X_1\looparrowright X_2},~(WAW)^{X_2\looparrowright X_1}\in A\{2^W\}$.
\end{proof}
\begin{lemma}
    Let $A\in \mathbb{C}^{m \times n}$ and let $X_1,~X_2\in A\{1^W\}$. 
    Then $(WAW)^{ X_1\looparrowright X_2}$ and $WAW)^{ X_2\looparrowright X_1}$ both belong to $A\{1^W\}$.
\end{lemma}
\begin{proof}
The assumption $X_1,~X_2\in A\{1^W\}$ leads to
    \begin{align*}
        AW (WAW)^{X_1\looparrowright X_2} WA &=(AWX_{1}WA)WX_{2}WA= AWX_{2}WA=A,\\
        AW (WAW)^{X_2\looparrowright X_1} WA &=AWX_{2}W(AWX_{1}WA)=AWX_{2}WA=A,
    \end{align*}
which implies $(WAW)^{X_1\looparrowright X_2},~(WAW)^{X_2\looparrowright X_1} \in A\{1^W\}$.
\end{proof}

Next, we present the following two results on characterization of weighted generalized bilateral inverses.

\begin{theorem}\label{wgthm}
    Let $A\in \mathbb{C}^{m \times n}$ and let $X_1\in A\{2^W\}$ and $X_2\in A\{1^W\}$. 
    Then the system of equations
    \begin{equation} \label{uniq}
        XWAWX=X,~XWA=X_1WA\text{ and } AWXWAWX=AW (WAW)^{X_1\looparrowright X_2}
    \end{equation} 
is consistent and has $X=(WAW)^{X_1\looparrowright X_2}$ as its unique solution.
\end{theorem}
\begin{proof}
    Suppose $X=(WAW)^{X_1\looparrowright X_2}$. 
An application of Lemma \ref{outer} leads to $XWAWX=X$. 
Further, 
\begin{align*}
        XWA&=X_1W(AWX_{2}WA)=X_1WA,\\
        AWXWAWX&=AWX_{1}W(AWX_{2}WA)WX_{1}WAWX_{2}\\&=AW(X_1WAWX_{1})WAWX_{2}\\
        &=AWX_{1}WAWX_{2}\\
        &= AW(WAW)^{X_1\looparrowright X_2}.
    \end{align*}
    To show uniqueness, assume that $Y$ and $Z$ both satisfy \eqref{uniq}. Thus
    \begin{align*}
        Y&=YWAWY=YWAWYWAWY\\
        &=YWAWX_{1}WAWX_{2}=X_1WAWX_{1}WAWX_{2}\\&=ZWAWX_{1}WAWX_{2}=ZWAWZ\\
        &=Z.
    \end{align*}
\end{proof}

\begin{theorem}
    Let $A\in \mathbb{C}^{m \times n}$ and let $X_1\in A\{2^W\}$ and $X_2\in A\{1^W\}$. 
    Then the system of equations
\[ XWAWX=X,~AWX=AWX_{1}\text{ and } XWAWXWA=AW (WAW)^{X_2\looparrowright X_1}\] 
is consistent and has $X=(WAW)^{X_2\looparrowright X_1}$ as its unique solution.
\end{theorem}
\begin{proof}
    The proof is similar to Theorem\ref{wgthm}.
\end{proof}

\subsection{Dual weighted generalized bilateral inverse}
In this subsection, we present the definition of the dual of weighted generalized bilateral inverses.
\begin{definition}
    Let $A\in \mathbb{C}^{m \times n}$ and $W\in \mathbb{C}^{n \times m}$. Let $X_1,X_2\in A\{1^W\}\cup A\{2^W\}$.
Then the expression $(WAW)^{X_2\looparrowright X_1} :=X_2WAWX_{1}$ is called the dual weighted generalized bilateral inverse of $(WAW)^{X_1\looparrowright X_2} :=X_1WAWX_{2}$.
 If the condition $(WAW)^{X_1\looparrowright X_2} =(WAW)^{X_2\looparrowright X_1} $ is satisfied then $(WAW)^{X_1\looparrowright X_2} $ is called self dual.
\end{definition}
Let $A\in \mathbb{C}^{m \times n}$ and $W\in \mathbb{C}^{n \times m}$.
For $X_1\in A\{2^W\}\text{ and }X_2\in A\{1^W\}$, necessary and sufficient criteria for the self-duality of a generalized bilateral inverse $(WAW)^{X_1\looparrowright X_2} :=X_1WAWX_{2}$ are presented in Theorem \ref{ThmSelfD}.

\begin{theorem}\label{ThmSelfD}
    Let $A\in \mathbb{C}^{m \times n}$ and let $X_1\in A\{2^W\}$ and $X_2\in A\{1^W\}$. Then the following equations are equivalent:
    \begin{enumerate}
        \item[\rm (i)] $(WAW)^{X_1\looparrowright X_2}$ is self dual;
        \item[\rm (ii)] $X_1=(WAW)^{X_1\looparrowright X_2} =(WAW)^{X_2\looparrowright X_1} $;
        \item[\rm (iii)] $\nl(WAWX_{2})\subseteq \nl(X_1)\text{ and }\rg(X_1)\subseteq \rg(X_2WAW)$.
    \end{enumerate}
\end{theorem}
\begin{proof}
    ${\rm (i)}\Rightarrow {\rm (ii)}$ Suppose (i) holds, i.e., $(WAW)^{X_1\looparrowright X_2} =(WAW)^{X_2\looparrowright X_1} $.
Since $AWX_{2}WA=A$ and $X_1WAWX_{1}=X_1$, it follows
\begin{align*}
    X_1&=X_1WAWX_{1}=X_1W(AWX_{2}WA)WX_{1}\\
    &=(WAW)^{X_1\looparrowright X_2}WAWX_{1}\\&=
   (WAW)^{X_2\looparrowright X_1}WAWX_{1}=X_2WAW(X_1WAWX_{1})\\
   &=(WAW)^{X_2\looparrowright X_1}.
\end{align*}
Hence $X_1=(WAW)^{X_1\looparrowright X_2}$.

\smallskip
\noindent ${\rm (ii)}\Rightarrow {\rm (iii)}$ Suppose {\rm (ii)} holds; then $\nl(WAWX_{2})\subseteq \nl(X_1WAWX_{2})=\nl(X_1)$.
Furthermore, using  $X_1=(WAW)^{X_2\looparrowright X_1}$, we get that $\rg(X_1)=\rg(X_2WAWX_{1})\subseteq \rg(X_2WAW)$.

\smallskip
\noindent ${\rm (iii)}\Rightarrow {\rm (i)}$ It is clear that $WAWX_{2}$ is a projection which implies $\rg(I_{n}-WAWX_{2})=\nl(WAWX_{2})$. 
Since $\nl(WAWX_{2})\subseteq \nl(X_1)$, it follows $\rg(I_{n}-WAWX_{2})\subseteq \nl(X_1)$ and hence $X_1(I_{n}-WAWX_{2})=0$. Hence, $X_1=X_1WAWX_{2}$.
Also we know $X_2WAW$ is a projection which implies $\rg(X_2WAW)=\nl(I_{m}-X_2WAW)$. 
Since $\rg(X_1)\subseteq \rg(X_2WAW)=\nl(I_{m}-X_2WAW)$, we obtain that $(I_{m}-X_2WAW)X_1=0$. 
Hence $X_1=(WAW)^{X_2\looparrowright X_1}$ and the result is confirmed.
\end{proof}

\begin{theorem}
    Let $A\in \mathbb{C}^{m \times n}$ and let $X_1,~X_2\in A\{1^W\}$. Then the following equations are equivalent:
    \begin{enumerate}
        \item[\rm (i)] $(WAW)^{X_1\looparrowright X_2}=(WAW)^{X_2\looparrowright X_1}$;
        \item[\rm (ii)] $AWX_{2}=AWX_{1}$ and $X_1WA=X_2WA$.
    \end{enumerate}
\end{theorem}
\begin{proof}
    ${\rm (i)}\Rightarrow {\rm (ii)}$ Since $X_1WAWX_{2}=X_2WAWX_{1}$. By using $AWX_{1}WA=A$ and $AWX_{2}WA=A$ and pre-multiplying by $AW$, we get $$AWX_{2}=AWX_{1}WAWX_{2}=AWX_{2}WAWX_{1}=AWX_{1}.$$ Similarly by post-multiplying $WA$, we get $X_1WA=X_2WA$.

\smallskip
    ${\rm (ii)}\Rightarrow {\rm (i)}$ If $AWX_{2}=AWX_{1}$, then by pre-multiplying $X_1W$, we get $X_1WAWX_{2}=X_1WAWX_{1}$.
    Again, post-multiplying $WX_{1}$ by $X_1WA=X_2WA$, we get $X_1WAWX_{1}=X_2WAWX_{1}$. Hence, $X_1WAWX_{2}=X_2WAWX_{1}$.
\end{proof}
\begin{proposition}\label{rangeNull}
    Let $A\in \mathbb{C}^{m \times n}$ and let $X\in A\{2^W\}$. Then $\nl(AWX)=\nl(XWAWX)=\nl(X)$ and $\rg(XWA)=\rg(XWAWX)=\rg(X)$.
\end{proposition}
In the following theorem, we present necessary and sufficient condition for self-duality of a weighted generalized bilateral inverses.
\begin{theorem}
    Let $A\in \mathbb{C}^{m \times n},~X_2\in G_{ow}$ and $X_1\in A\{1^W\}\cup A\{2^W\}$. Then the following statements are equivalent:
    \begin{enumerate}
        \item[\rm (i)] $(WAW)^{X_1\looparrowright X_2}=(WAW)^{X_2\looparrowright X_1}$;
        \item[\rm (ii)] $\nl(X_2)\subseteq \nl((WAW)^{X_2\looparrowright X_1})$ and $\rg((WAW)^{X_1\looparrowright X_2}\subseteq \rg(X_2)$.
    \end{enumerate}
\end{theorem}
\begin{proof}
    ${\rm (i)}\Rightarrow {\rm (ii)}$ Assume $(WAW)^{X_1\looparrowright X_2}=(WAW)^{X_2\looparrowright X_1}$ and $X_2WAWX_{2}=X_2$. Since $WAWX_{2}$ and $X_1WAW$ are projections. 
Then it follows \\
$\nl(WAWX_{2})=\rg(I_{n}-WAWX_{2})$ and $\nl(I_{m}-X_2WAW)=\rg(X_2WAW)$.\\ 
Also the following statements hold:
     \begin{align}\label{12}\notag
       (WAW)^{X_2\looparrowright X_1}&=(X_2WAWX_{2})WAWX_{1}=(WAW)^{X_2\looparrowright X_1}WAWX_{2},\\
        & (WAW)^{X_2\looparrowright X_1}(I_{n}-WAWX_{2})=0,\\
        &\notag \rg(I_{n}-WAWX_{2})\subseteq \nl((WAW)^{X_2\looparrowright X_1}),
    \end{align}
and
    \begin{align}\label{13}\notag
    (WAW)^{X_1\looparrowright X_2}&=X_1WAW(X_2WAWX_{2})=X_2WAW (WAW)^{X_1\looparrowright X_2},\\
    &(I_{m}-X_2WAW)(WAW)^{X_1\looparrowright X_2}=0,\\ 
    &\notag \rg((WAW)^{X_1\looparrowright X_2})\subseteq \nl(I_{m}-X_2WAW).
    \end{align}
Therefore,
\begin{equation}\label{14}
    \nl(WAWX_{2})\subseteq \nl((WAW)^{X_2\looparrowright X_1}),\ ~\rg((WAW)^{X_1\looparrowright X_2})\subseteq \rg(X_2WAW).
\end{equation}
Using Proposition \ref{rangeNull}, we obtain $\nl(X_2)= \nl(WAWX_{2})$ and $\rg(X_2)= \rg(X_2WA)$.
Hence,
\begin{equation}\label{15}
    \nl(X_2)\subseteq \nl((WAW)^{X_2\looparrowright X_1}),\ \ \rg((WAW)^{X_1\looparrowright X_2}) \subseteq \rg(X_2).
\end{equation}
${\rm (ii)}\Rightarrow {\rm (i)}$ If (ii) holds, then from the equivalencies of \eqref{12},\eqref{13},\eqref{14}, and \eqref{15}, we obtain
$$
\aligned
(WAW)^{X_2\looparrowright X_1}&=(WAW)^{X_2\looparrowright X_1}WAWX_{2},\\ 
(WAW)^{X_1\looparrowright X_2}&=X(WAW)^{X_2\looparrowright X_1}WAWX_{2},
\endaligned
$$
which was our intention.
\end{proof}


\noindent{\bf Funding}
\begin{itemize}
\item Ratikanta Behera is supported by the Science and Engineering Research Board (SERB), Department of Science and Technology, India, under Grant No. EEQ/2022/001065.
\item Jajati Keshari Sahoo is supported by the Science and Engineering Research Board (SERB), Department of Science and Technology, India, under Grant No. SUR/2022/004357.
\item Predrag Stanimirovi\' c is supported by Science Fund of the Republic of Serbia (No. 7750185, Quantitative Automata Models: Fundamental Problems and Applications - QUAM). \\
Further, Predrag Stanimirovi\' c is supported by the Ministry of Science and Higher Education of the Russian Federation (Grant No. 075-15-2022-1121).
\end{itemize}

\medskip
\bibliographystyle{abbrv}
\bibliography{Reference}
\end{document}